\documentclass[a4paper,12pt,reqno]{amsart}
\usepackage{amsmath, amssymb,amsthm}
\usepackage{enumerate}
\usepackage{cite}
\usepackage{graphicx}
\usepackage{subcaption}

\numberwithin{equation}{section}
\renewcommand{\vec}[1]{\boldsymbol{#1}}

\newtheorem{theorem}{Theorem}[section]
\newtheorem{lemma}{Lemma}[section]

\newtheorem{corollary}{Corollary}[section]
\newtheorem{remark}{Remark}[section]

\theoremstyle{definition}

\begin{document}

\title[Biorthogonality and para-orthogonality of $R_I$ polynomials]
{Biorthogonality and para-orthogonality of $R_I$ polynomials}


\author[K. K. Behera]{Kiran Kumar Behera}
\address{
Department of Mathematics,
Indian Institute of Technology, Roorkee-247667,
Uttarakhand, India}
\curraddr{}
\email{krn.behera@gmail.com}
\thanks{}

\author[A. Swaminathan]{A. Swaminathan}
\address{
Department of Mathematics,
Indian Institute of Technology, Roorkee-247667,
Uttarakhand, India}
\curraddr{}
\email{mathswami@gmail.com, swamifma@iitr.ac.in}
\thanks{}



\date{}

\dedicatory{}

\commby{}

\begin{abstract}
In this paper, a sequence of linear combination of $R_{I}$
polynomials such that the terms in this
sequence have a common zero is constructed.
A biorthogonality relation arising from such a sequence is discussed.
Besides, a sequence of para-orthogonal polynomials
by removing the common zero
using suitable conditions is obtained.
Finally, a case of hypergeometric functions is studied
to illustrate the results obtained.
\end{abstract}

\maketitle

\pagestyle{myheadings}
\markboth
{K. K. Behera and A. Swaminathan}
{Biorthogonality and para-orthogonality of $R_{I}$ polynomials}

\section{Introduction}
Recurrence relations of the form
\begin{align}
\label{eqn: ttrr for RI polynomials Pn}
\mathcal{P}_{n+1}(\lambda)=\rho_n(\lambda-\beta_n)\mathcal{P}_{n}(\lambda)+
\tau_n(\lambda-\gamma_n)\mathcal{P}_{n-1}(\lambda),
\quad n\geq1,
\end{align}
with $\mathcal{P}_1(\lambda)=\rho_0(\lambda-\beta_0)$
and $\mathcal{P}_0(\lambda)=1$ are well studied
\cite{Ismail-Masson-generalized-orthogonality-JAT-1995}.
Here, $\lambda$ is a real or complex variable.
In addition to the restrictions
$\rho_n\neq0$ and $\tau_n\neq0$, $n\geq0$,
it was shown that if
one also assumes $\mathcal{P}_{n}(\gamma_{n})\neq0$, $n\geq1$, in
\eqref{eqn: ttrr for RI polynomials Pn}, then there exists a linear functional
$\mathcal{M}$ such that the orthogonality relations
\begin{align*}
\mathcal{M}[\gamma_0]\neq0,
\quad
\mathcal{M}\left[\frac{\lambda^{k}}
{\prod_{k=1}^{n}(\lambda-\gamma_k)}
\mathcal{P}_n(\lambda)\right]
\neq0,\quad
0\leq k<n,
\end{align*}
hold \cite[Theorem 2.1]{Ismail-Masson-generalized-orthogonality-JAT-1995}.
Following
\cite{Ismail-Masson-generalized-orthogonality-JAT-1995},
the recurrence relation
\eqref{eqn: ttrr for RI polynomials Pn} will be referred as
recurrence relation of $R_{I}$ type and
$\mathcal{P}_n(\lambda)$, $n\geq1$,
generated by it as
$R_{I}$ polynomials.

A non-trivial positive measure of orthogonality defined
either on the unit circle or
on a subset of the real axis is associated with
\eqref{eqn: ttrr for RI polynomials Pn}
whenever
$\gamma_n=0$, $n\geq1$
and the parameters satisfy specific conditions.
For instance, if $\rho_n>0$, $\beta_n>0$ and $\tau_n<0$,
then the corresponding polynomials satisfy the
Laurent orthogonality property
\cite{Jones-Thron-strong-stieltjes-moment-AMS-1980}
(see also \cite{Silva-Ranga-bounds-complex-zeros-JAT-2005})
\begin{align*}
\int_{0}^{\infty}t^{-n+s}\mathcal{P}_n(t)d\phi(t)=0,
\quad s=0,1,\cdots, n-1.
\end{align*}
Similarly, while $\rho_n=1$, $\beta_n\neq0$ and $\tau_n\neq0$,
there exists a positive measure $\mu$ on the unit circle such that
$\{\mathcal{P}_n(\lambda)\}_{n=1}^{\infty}$
is a sequence of Szeg\H{o} polynomials
\cite[Theorem 2.1]{Ranga-szego-polynomials-2010-AMS}
whenever $0<\tau_n\beta_n^{-1}<1-|\mathcal{P}_n(0)|^2$, $n\geq1$.
The theory of polynomials orthogonal on the unit circle is a classical one.
For basic results in this direction, we
refer to the monographs
\cite{Simon-book-Part-1, Szego-book}
and references therein.

Results on linear combination of orthogonal polynomials are abundant in
literature and are studied in the context of quasi-orthogonality. A polynomial
sequence $\{q_n(x)\}_{n=0}^{\infty}$ is said to be
quasi-orthogonal of order $r$ on $[a,b]\subseteq\mathbb{R}$
if it satisfies the orthogonality conditions
\cite{Brezinski-Driver-ANM-2004-quasi-orthogonality}
\begin{align}
\label{eqn: definition of quasi orthogonality}
\int_{a}^{b}x^kq_n(x)\omega(x)dx\left\{
                                  \begin{array}{ll}
                                    =0, & \hbox{$k=0,\cdots,n-r$;} \\
                                    \neq 0, & \hbox{$k=n-r$.}
                                  \end{array}
                                \right.
\end{align}
A detailed exploration of the concept of quasi-orthogonality was first made
in 
\cite{Riesz-quasi-orthogonality} (see also
\cite{Chihara-quasi-orthogonalit-AMS-1957})
while studying the solution of the Hamburger moment problem.
The case $r=1$ \cite{Fejer-quasi-orthogonality-1933}
and the general case
\cite{Shohat-quasi-orthogonality-AMS-1937}
were studied extensively, with their algebraic properties
being investigated by several authors
\cite{Chihara-quasi-orthogonalit-AMS-1957,
Dickinson-quasi-orthogonality-AMS-1961,
Draux-quasi-order-r-ITSF-2016}.
Applications of quasi-orthogonality to families of classical orthogonal
polynomials, in particular, to the location of their zeros
are also investigated
\cite{Brezinski-Driver-ANM-2004-quasi-orthogonality}.

A necessary and sufficient condition
\cite{Chihara-quasi-orthogonalit-AMS-1957}
for $q_n(x)$ of degree $n$ to be quasi-orthogonal
of order $r$ with respect to $\omega(x)$ on $[a,b]$
is that there exists another sequence of
polynomials $\{p_n(x)\}_{n=0}^{\infty}$ orthogonal with respect to $\omega(x)$
on $[a,b]$ such that
\begin{align*}
q_n(x)=c_0p_n(x)+c_1p_{n-1}(x)+\cdots+c_rp_{n-r}(x),
\end{align*}
where $c_i$ depend only on $n$ and $c_0c_r\neq0$.
Linear combination of polynomials which are
orthogonal either on the real line or on the unit circle has also been studied
as an independent problem
\cite{Marcellan-linear-combinations-1996-advances,
Alfaro-linear-combinations-JCAM-2010}, with conditions being
obtained for the orthogonality of such linear combinations.
We note that if $r=0$ in \eqref{eqn: definition of quasi orthogonality},
we recover the orthogonality condition for the polynomial $\{p_n(x)\}_{n=0}^{\infty}$
on the real line.

Motivated by linear combinations of polynomials and
their orthogonality as well as algebraic properties,
our aim in the present paper
is to study the linear combination
of two successive $R_I$ polynomials of a sequence
$\{\mathcal{P}_n(\lambda)\}_{n=0}^{\infty}$
that satisfies
\eqref{eqn: ttrr for RI polynomials Pn}
such that
\begin{align*}
  \mathcal{Q}_n(\lambda):=
  \mathcal{P}_n(\lambda)+\alpha_n\mathcal{P}_{n-1}(\lambda),
  \quad
  \alpha_n\in\mathbb{R}\setminus\{0\},
  \quad n\geq0,
\end{align*}
where $\beta_0\neq0,\pm1$ and $\beta_n\neq0$, $n\geq1$.
We construct a unique sequence $\{\alpha_n\}_{n=0}^{\infty}$
such that
$\{\mathcal{Q}_n(\lambda)\}_{n=0}^{\infty}$
not only satisfies mixed recurrence relations of
$R_{I}$ and $R_{II}$ type but also has a common zero.
Note that
\begin{align*}
\mathcal{\hat{P}}_{n+1}(\lambda)=
\hat{\rho}_n(\lambda-\hat{\beta}_n)\mathcal{\hat{P}}_{n}(\lambda)+
\hat{\tau}_n(\lambda-\hat{\gamma}_n^{(1)})(\lambda-\hat{\gamma}_n^{(2)})
\mathcal{\hat{P}}_{n-1}(\lambda),
\quad n\geq1,
\end{align*}
with $\mathcal{\hat{P}}_1(\lambda)=\hat{\rho}_0(\lambda-\hat{\beta}_0)$ and
$\mathcal{\hat{P}}_0(\lambda)=1$
is called a recurrence relation of $R_{II}$ type
if $\hat{\rho}_n\neq0$ and $\hat{\tau}_n\neq0$ for $n\geq0$
\cite{Ismail-Masson-generalized-orthogonality-JAT-1995}.
Common zeros of an orthogonal sequence have been considered in the past,
see for example
\cite{Driver-Muldoon-common-zeros-Laguerre-JAT-2015,
Gibson-common-zeros-JAT-2000,
Wong-1st-2nd-kind-POP-JAT-2007}.
However, the novelty in our approach is that we actually construct
such a sequence before studying its orthogonality properties.

The polynomials $\mathcal{Q}_n(\lambda)$, $n\geq1$,
are shown to satisfy biorthogonality relations
that are obtained from their eigenvalue representations.
Two polynomial sequences $\{\phi_n(\lambda)\}_{n=1}^{\infty}$ and
$\{\psi_n(\lambda)\}_{n=1}^{\infty}$
are said to biorthogonal with respect to a moment functional
$\mathfrak{N}$ if the orthogonality relations
\begin{align*}
\mathfrak{N}(\phi_n(\lambda)\psi_m(\lambda))=\kappa_n\delta_{n,m},
\quad \kappa_n\neq0,
\quad n,m\geq0,
\end{align*}
hold, where in contrast to orthogonal polynomials,
two different sequences of polynomials are used for the
biorthogonality condition.

With certain additional conditions, we also show that a
para-orthogonal polynomial of degree $n$ can be obtained from
$\mathcal{Q}_{n+1}(\lambda)$.
The para-orthogonal polynomials introduced in
\cite{Jones-Njasad-Thron-Moment-OP-CF-1989-BLMS}
have their zeros on the unit circle and are used to
construct numerical quadrature formula on the unit circle.

The rest of the paper is organised as follows.
In Section \ref{sec: Section-2-mixed recurrence relation}
we obtain the mixed recurrence relations using which we show that
$\mathcal{Q}_n(\lambda)$ has a common zero for $n\geq1$.
Section \ref{sec: orthogonality relations from linear combinations} gives
biorthogonality relations while in
Section \ref{sec: para-orthogonality relations} we obtain a para-orthogonal
polynomial $\mathcal{R}_n(\lambda)$ from $\mathcal{Q}_{n+1}(\lambda)$.
We provide an illustration of polynomials of  Gaussian hypergeometric form
leading to the same para-orthogonal polynomial
having two different representations in Section \ref{sec: Section-3-illustration}.

\section{Mixed recurrence relations of $R_I$ and $R_{II}$ type}
\label{sec: Section-2-mixed recurrence relation}
It is known \cite{Draux-quasi-order-r-ITSF-2016,
Dickinson-quasi-orthogonality-AMS-1961}
that quasi-orthogonal polynomials satisfy
recurrence relations with polynomial coefficients.
In some other cases, these recurrence relations are of mixed type,
depending on the parameter involved
and are used to obtain information on the zeros of the
polynomials satisfying such relations
\cite{Jordaan-mixed-recurrence-Acta-Hungarica-2010,
Koepf-mixed-recurrence-ANM-2018}.

The first result shows that $\mathcal{Q}_n(\lambda)$,
$n\geq1$
satisfies a three term recurrence relation with polynomial
coefficients of degree at most two.
\begin{theorem}
\label{thm: ttrr for Q-n}
  Given a sequence $\{\mathcal{P}_n(\lambda)\}_{n=1}^{\infty}$
  of polynomials satisfying $R_I$ type recurrence relations
  \eqref{eqn: ttrr for RI polynomials Pn},
  consider the linear combinations of two successive such polynomials
  \begin{align}
  \label{eqn: Q-n as linear combination in theorem}
  \mathcal{Q}_n(\lambda)=
  \mathcal{P}_n(\lambda)+\alpha_n\mathcal{P}_{n-1}(\lambda),
  \quad \alpha_n\in\mathbb{R}\setminus\{0\},
  \quad n\geq0,
  \end{align}
  where $\mathcal{Q}_0(\lambda)=\mathcal{P}_0(\lambda)=1$.
  Then there exist constants
  $\{p_n, q_n, r_n, s_n, t_n, u_n, v_n, w_n\}$ such that
  $\{\mathcal{Q}_n(\lambda)\}_{n=1}^{\infty}$
  satisfies a three term recurrence relation of the form
  \begin{align}
  \label{eqn: ttrr for Q-n}
  \lefteqn{(p_n\lambda+q_n)\mathcal{Q}_{n+1}(\lambda)}\nonumber\\
  &&=
  (r_n\lambda^2+s_n\lambda+t_n)
  \mathcal{Q}_n(\lambda)+(u_n\lambda^2+v_n\lambda+w_n)
  \mathcal{Q}_{n-1}(\lambda), \quad n\geq1,
  \end{align}
  with $\mathcal{Q}_1(\lambda)=
  \rho_0(\lambda+\alpha_1\rho_0^{-1}-b_0)$ and $\mathcal{Q}_0(\lambda)=1$.
  \end{theorem}
\begin{proof}
  For $n\geq1$, consider the following system
  \begin{align}
\begin{split}
  \label{eqn: system of equations Qn and Pn for recurrence relations}
  \mathcal{Q}_{k}(\lambda)
&=\mathcal{P}_{k}(\lambda)+\alpha_{k}\mathcal{P}_{k-1}(\lambda),
  \quad k=n-1,n,n+1,\\
  \mathcal{P}_{k}(\lambda)
&=\rho_{k-1}(\lambda-\beta_{k-1})\mathcal{P}_{k-1}(\lambda)+
  \tau_{k-1}(\lambda-\gamma_{k-1})\mathcal{P}_{k-2}(\lambda),
  \quad k=n,n+1,
\end{split}
  \end{align}
  written as
  \begin{align*}
  [\mathcal{C}_{n-1}]
  \left(
    \begin{array}{c}
      \mathcal{Q}_{n-1}(\lambda) \\
      \mathcal{P}_{n-2}(\lambda) \\
      \mathcal{P}_{n-1}(\lambda) \\
      \mathcal{P}_{n}(\lambda) \\
      \mathcal{P}_{n+1}(\lambda) \\
    \end{array}
  \right)
  =
  \left(
    \begin{array}{c}
      \mathcal{Q}_{n+1}(\lambda) \\
      \mathcal{P}_{n}(\lambda) \\
      0 \\
      0 \\
      0 \\
    \end{array}
  \right),
  \end{align*}
 where
 \begin{align*}
 [\mathcal{C}_{n-1}]=
\left(
    \begin{array}{ccccc}
      0 & 0 & 0 & \alpha_{n-1} & 1 \\
      0 & 0 & \alpha_n & 1 & 0 \\
      -1 & \alpha_{n-1} & 1 & 0 & 0 \\
      0 & 0 & \tau_n(\lambda-\gamma_n) & \rho_n(\lambda-\beta_n) & -1 \\
      0 & \tau_{n-1}(\lambda-\gamma_{n-1}) &
      \rho_{n-1}(\lambda-\beta_{n-1}) & -1 & 0 \\
    \end{array}
  \right) \end{align*}
is the coefficient matrix.
Using Cramer's rule,
the first unknown variable $\mathcal{Q}_{n-1}(\lambda)$
is given by
\begin{align*}
   \det[\mathcal{C}_{n-1}]\mathcal{Q}_{n-1}(\lambda)=
  \det[\mathcal{A}_{n+1}]\mathcal{Q}_{n+1}(\lambda)-
  \det[\mathcal{B}_{n}]\mathcal{Q}_{n}(\lambda),
\end{align*}
  where
  \begin{align*}
 [\mathcal{A}_{n+1}]=
 \left(
    \begin{array}{cccc}
      0            & \alpha_n & 1 & 0 \\
      \alpha_{n-1} & 1     & 0 & 0  \\
      0            & \tau_n(\lambda-\gamma_n)
      & \rho_n(\lambda-\beta_n) & -1 \\
      \tau_{n-1}(\lambda-\gamma_{n-1})
      & \rho_{n-1}(\lambda-\beta_{n-1}) & -1 & 0 \\
      \end{array}
  \right),
 \end{align*}
 \begin{align*}
 [\mathcal{B}_{n}]=
 \left(
    \begin{array}{cccc}
      0            & 0 & \alpha_{n+1} & 1 \\
      \alpha_{n-1} & 1     & 0 & 0  \\
      0            & \tau_n(\lambda-\gamma_n)
      & \rho_n(\lambda-\beta_n) & -1 \\
      \tau_{n-1}(\lambda-\gamma_{n-1})
      & \rho_{n-1}(\lambda-\beta_{n-1}) & -1 & 0 \\
      \end{array}
  \right).
 \end{align*}
A straightforward computation of the determinants gives
 $ \det[\mathcal{A}_{n+1}] = p_n\lambda+q_n$,
 $\det[\mathcal{B}_{n}]  = r_n\lambda^2+s_n\lambda+t_n$
 and
 $\det[\mathcal{C}_{n-1}] = u_n\lambda^2+v_n\lambda+w_n$,
where for $n\geq1$,
\begin{align}
\label{eqn: ttrr constants in terms of RI parameters}
  p_n&=\alpha_{n-1}\rho_{n-1}-\tau_{n-1},\,\,
  q_n=\alpha_{n-1}(\alpha_n-\rho_{n-1}\beta_{n-1})+
  \tau_{n-1}\gamma_{n-1},
 \nonumber\\
  r_n&=\rho_np_n,\,\,
  s_n=\rho_nq_n+\alpha_n^{-1}p_nq_{n+1}-\alpha_{n-1}p_{n+1},
  \nonumber\\
  t_n&=-\alpha_{n-1}\rho_{n-1}\beta_{n-1}(\alpha_{n+1}-\rho_n\beta_n)+
           \gamma_{n-1}(\alpha_{n-1}\alpha_{n+1}-\tau_{n-1}\beta_n\rho_n)
           -\alpha_{n-1}\tau_n\gamma_n,
           \nonumber\\
  u_n&=\tau_{n-1}p_{n+1},\,\,
  v_n=\tau_{n-1}(q_{n+1}-\gamma_{n-1}p_{n+1}),\,\,
 w_n=-\tau_{n-1}\gamma_{n-1}q_{n+1},\nonumber
  \end{align}
  are the constants given explicitly in terms of
  the recurrence parameters used in
\eqref{eqn: ttrr for RI polynomials Pn} and
$\{\alpha_n\}_{n=0}^{\infty}$.
Further, the values of $\mathcal{Q}_k(\lambda)$, $k=0,1$
are obtained from
\eqref{eqn: system of equations Qn and Pn for recurrence relations}
for $n=0,1$ and
hence the recurrence relation
\eqref{eqn: ttrr for Q-n} is well-defined.
\end{proof}
An immediate consequence is the following.
\begin{corollary}
\label{coro: corollary for choice of alpha-n for R-I polynomials}
The polynomials
$\mathcal{Q}_n(\lambda)=
\mathcal{P}_n(\lambda)+\alpha_n\mathcal{P}_{n-1}(\lambda)$,
$n\geq0$,
form a sequence of $R_I$ polynomials if
$\alpha_n=\rho_{n-1}\beta_{n-1}$ and
$\gamma_n=0$, $n\geq0$ or
$\alpha_n=\tau_n\rho_n^{-1}$, $n\geq0$.
\end{corollary}
\begin{proof}
Choosing $\alpha_n=\tau_n\rho_n^{-1}$, $n\geq0$, makes
 $p_n$ and hence $r_n$ and $u_n$ equal to zero for $n\geq1$.
Similarly, with $\alpha_n=\rho_{n-1}\beta_{n-1}$ and $\gamma_n=0$
we have $t_n=w_n=0$ for $n\geq1$.
Thus \eqref{eqn: ttrr for Q-n} reduces to the recurrence relations
\begin{align*}
\mathcal{Q}_{n+1}(\lambda)
&=
q_n^{-1}(s_n\lambda+t_n)\mathcal{Q}_{n}(\lambda)+
q_n^{-1}(v_n\lambda+w_n)\mathcal{Q}_{n-1}(\lambda),
\quad n\geq1,\\
\mathcal{Q}_{n+1}(\lambda)&=
p_n^{-1}(r_n\lambda+s_n)\mathcal{Q}_{n}(\lambda)+
p_n^{-1}(u_n\lambda+v_n)\mathcal{Q}_{n-1}(\lambda),
\quad n\geq1,
\end{align*}
of $R_I$ type respectively.
\end{proof}
It is clear that there is an obvious way to choose $\alpha_n$, $n\geq0$,
if we require such linear combinations to be $R_I$ polynomials \textit{ab initio}.
We will use this choice in Section
\ref{sec: para-orthogonality relations} when we
obtain a para-orthogonal polynomial from $\mathcal{Q}_n(\lambda)$.
However, in the present section, we suppose
$p_{n+1}\neq 0$, $q_n\neq 0$,
$\gamma_n=\gamma\in\mathbb{C}$, $n\geq1$.
In such a case, the polynomials $\mathcal{P}_n(\lambda+\gamma)$, $n\geq1$,
satisfy the recurrence relation
\begin{align*}
\mathcal{P}_{n+1}(\lambda)=\rho_n(\lambda-\beta_n)\mathcal{P}_{n}(\lambda)+
\tau_n\lambda\mathcal{P}_{n-1}(\lambda),
\quad n\geq1.
\end{align*}
Then, from Theorem $\ref{thm: ttrr for Q-n}$,
the linear combination
\begin{align*}
\mathcal{Q}_{n}(\lambda)=\mathcal{P}_n(\lambda+\gamma)+
a_n\mathcal{P}_{n-1}(\lambda+\gamma),
\quad n\geq1,
\end{align*}
 satisfies the recurrence relation \eqref{eqn: ttrr for Q-n}
but with the much simplified constants
\begin{equation}
\label{eqn: much simplified constants}
\begin{split}
p_n
&=
\alpha_{n-1}\rho_{n-1}-\tau_{n-1},
q_n=\alpha_{n-1}(\alpha_n-\rho_{n-1}\beta_{n-1}),
r_n=\rho_np_n,\\
s_n
&=
\alpha_n^{-1}p_nq_{n+1}+\alpha_{n-1}(\tau_n-\rho_{n-1}\beta_{n-1}\rho_n),
t_n=-\alpha_n^{-1}\alpha_{n-1}\rho_{n-1}\beta_{n-1}q_{n+1},\\
u_n
&=
\tau_{n-1}p_{n+1},
v_n=\tau_{n-1}q_{n+1},
w_n=0.
\end{split}
\end{equation}
We use these simplified constants to convert \eqref{eqn: ttrr for Q-n}
into a form that is appropriate for further discussion.
\begin{theorem}
\label{thm: mixed recurrence relations}
  Suppose the sequence $\{\alpha_n\}_{n=1}^{\infty}$ is constructed
recursively as
  \begin{align}
  \alpha_n=-(\rho_{n-1}-\alpha_{n-1}^{-1}\tau_{n-1})+\rho_{n-1}\beta_{n-1},
  \quad n\geq2,
  \label{eqn: condition for mixed recurrence relation}
  \end{align}
 where $\alpha_1\neq\rho_0\beta_0$ is arbitrary.
 If $\alpha_0=\tau_0\rho_0^{-1}$, then
 $\{\mathcal{Q}_n(\lambda)\}_{n=1}^{\infty}$
 satisfies the mixed recurrence relations
\begin{subequations}
 \begin{align}
 \label{eqn: mixed ttrr for Q-2}
 \mathcal{Q}_2(\lambda)
&=
\frac{s_1}{q_1}\left(\lambda+\frac{t_1}{s_1}\right)
 \mathcal{Q}_1(\lambda)-
 \frac{\tau_0q_2}{q_1}\lambda(\lambda-1)\mathcal{Q}_0(\lambda),
 \quad\mbox{and}\\
\mathcal{Q}_{n+1}(\lambda)
&=
\rho_n\left(\lambda-\frac{t_n}{r_n}\right)
\mathcal{Q}_{n}(\lambda)+
\frac{\tau_{n-1}q_{n+1}}{q_n}\lambda\mathcal{Q}_{n-1}(\lambda),
\quad n\geq2,
\label{eqn: mixed ttrr for for Q-n}
\end{align}
\end{subequations}
with $\mathcal{Q}_1(\lambda)=
  \rho_0(\lambda+\alpha_1\rho_0^{-1}-\beta_0)$ and
$\mathcal{Q}_0(\lambda)=1$
\end{theorem}
\begin{proof}
 It immediately follows from
 \eqref{eqn: much simplified constants} and
\eqref{eqn: condition for mixed recurrence relation}
 that, for $n\geq 2$,
\begin{align*}
q_n&=-p_n, \quad r_n=\rho_np_n, \quad s_n=-\alpha_n^{-1}p_np_{n+1}-\alpha_{n-1}p_{n+1}-\rho_np_n\\
t_n&=\alpha_n^{-1}\alpha_{n-1}\rho_{n-1}\beta_{n-1}p_{n+1},
\quad u_n=-v_n=\tau_{n-1}p_{n-1},
\quad w_n=0.
\end{align*}
Then, for $n\geq2$, the recurrence relation \eqref{eqn: ttrr for Q-n} takes the form
\begin{align}
\label{eqn: ttrr for Q_n+1 after simplified constants in thm}
(\lambda-1)\mathcal{Q}_{n+1}(\lambda)=
p_n^{-1}[r_n\lambda^2+s_n\lambda+t_n]\mathcal{Q}_{n}(\lambda)
+
p_n^{-1}u_n\lambda(\lambda-1)\mathcal{Q}_{n-1}(\lambda).
\end{align}
Further, $p_n+q_n=0$ implies
\begin{align*}
\frac{p_{n+1}}{\alpha_n}+\frac{\alpha_{n-1}p_{n+1}}{\alpha_n}
\left(1-\frac{\rho_{n-1}\beta_{n-1}}{\alpha_n}\right)=0,
\quad n\geq2,
\end{align*}
which means that $\lambda-1$ is a factor of the polynomial
coefficient of $\mathcal{Q}_{n}(\lambda)$ in
\eqref{eqn: ttrr for Q_n+1 after simplified constants in thm}.
Hence, cancelling out the factor ($\lambda-1$),
\eqref{eqn: ttrr for Q_n+1 after simplified constants in thm}
can now be written as
\begin{align*}
\mathcal{Q}_{n+1}(\lambda)&=
\rho_n\left(\lambda-\frac{\alpha_{n-1}\rho_{n-1}\beta_{n-1}p_{n+1}}
{\alpha_n\rho_n p_n}\right)
\mathcal{Q}_{n}(\lambda)+
\frac{\tau_{n-1}p_{n+1}}{p_n}\lambda\mathcal{Q}_{n-1}(\lambda),
\end{align*}
which is
\eqref{eqn: mixed ttrr for for Q-n} for $n\geq2$.

Now with the conditions $\alpha_0=\tau_0\rho_0^{-1}$ and
$\alpha_1\neq \rho_0\beta_0$, we observe that
$p_1=0$ and $q_1\neq0$. Further
\begin{align*}
r_1=0,\,\,
s_1=\rho_1q_1-\alpha_0p_2,\,\, t_1=a_0\rho_0b_0p_2\alpha_1^{-1},\,\,
u_1=-v_1=\tau_0p_2,\,\, w_1=0.
\end{align*}
Thus, \eqref{eqn: ttrr for Q-n} takes the form
\begin{align*}
q_1\mathcal{Q}_2(\lambda)=(s_1\lambda+t_1)
\mathcal{Q}_1(\lambda)+
\tau_0p_2\lambda(\lambda-1)\mathcal{Q}_0(\lambda),
\end{align*}
which is
\eqref{eqn: mixed ttrr for Q-2}.
\end{proof}
\begin{remark}
Similar to
\eqref{eqn: condition for mixed recurrence relation},
one can also construct $\{\alpha_n\}_{n=0}^{\infty}$
with $q_n=p_n$, $n\geq2$.
In this case also, as in
\eqref{eqn: ttrr for Q_n+1 after simplified constants in thm},
$\lambda+1$ becomes a factor of both the coefficients of
$\mathcal{Q}_n(\lambda)$ and $\mathcal{Q}_{n-1}(\lambda)$
which by a simple computation leads to
\eqref{eqn: mixed ttrr for Q-2} and \eqref{eqn: mixed ttrr for for Q-n} only.
\end{remark}
\begin{theorem}
\label{thm: theorem for Qn and Qn-1 not having common zero}
Suppose $\mathcal{Q}_1(1)=0$.
Then $\mathcal{Q}_n(\lambda)$ has a common zero at
$\lambda=1$ for $n\geq2$, with
$\mathcal{Q}_2(\lambda)$ having a double zero at $\lambda=1$
if $\alpha_1\neq\rho_0b_0$ is a root of the quadratic equation
\begin{align}
\label{eqn: condition for double zero of Q2 at 1}
\rho_1x^2-\rho_0\beta_0\tau_1=0.
\end{align}
However, if $\mathcal{Q}_2(\lambda)$ does not have a double zero at
$\lambda=1$, then, $\mathcal{Q}_n(\lambda)$ and
$\mathcal{Q}_{n-1}(\lambda)$, $n\geq2$,
do not have a common zero except at $\lambda=1$.
\end{theorem}
\begin{proof}
If $\mathcal{Q}_1(\lambda)$ has a zero at $\lambda=1$, it follows from the recurrence relations
\eqref{eqn: mixed ttrr for Q-2} and \eqref{eqn: mixed ttrr for for Q-n} that $\lambda=1$ is a root of
$\mathcal{Q}_n(\lambda)$, $n\geq1$. To check whether $\mathcal{Q}_2(\lambda)$
has a double zero at $\lambda=1$, we note from \eqref{eqn: mixed ttrr for Q-2} that
\begin{align*}
q_1\mathcal{Q}'_2(\lambda)=
s_1\mathcal{Q}_1(\lambda)+(s_1\lambda+t_1)\mathcal{Q}'_1(\lambda)
+2\lambda\tau_0p_2-\tau_0p_2.
\end{align*}
Since $q_1\neq0$, it follows that $\mathcal{Q}_2(\lambda)$
has a double zero at $\lambda=1$ if
$(s_1+t_1)\rho_0+\tau_0p_2=0$, which, upon substitution of the values of
$s_1$, $t_1$ and $p_2$
from Theorem \ref{thm: mixed recurrence relations},
is equivalent to the condition
$\rho_1\alpha_1^2-\rho_0\beta_0\tau_1=0$.

Further, $\mathcal{Q}_1(0)=\alpha_1-\rho_0\beta_0\neq0$ and
$\mathcal{Q}_{n+1}(0)=t_nq_n^{-1}\neq0$, $n\geq1$.
Hence, the last part of the theorem follows from
\cite[Lemma 2.1]{Silva-Ranga-bounds-complex-zeros-JAT-2005}
applied to the recurrence relation
\eqref{eqn: mixed ttrr for for Q-n}.
\end{proof}
The uniqueness of the sequence $\{\alpha_n\}_{n=0}^{\infty}$
hence follows
from the observation $\mathcal{Q}_1(1)=0$ implies choosing
$\alpha_1=\rho_0(\beta_0-1)$.
In case we require that $\mathcal{Q}_2(\lambda)$ does not have a
repeated root at $\lambda=1$, we need to choose the initial
values such that
$\tau_1\rho_1^{-1}\neq \rho_0\beta_0^{-1}(1-\beta_0)^{2}$,
$\beta_0\neq 0,\pm1$.
In the rest of the paper, we will assume that such a choice for
$\rho_0$, $\beta_0$, $\rho_1$ and $\tau_1$
has been made to satisfy
these inequalities.

\section{Biorthogonality relations from linear combinations}
\label{sec: orthogonality relations from linear combinations}
In this section, we consider the sequence
$\{\mathcal{Q}_n(\lambda)\}_{n=1}^{\infty}$ constructed from
\eqref{eqn: condition for mixed recurrence relation}
with $\alpha_1=\rho_0(\beta_0-1)$ so that
$\mathcal{Q}_n(1)=0$ for $n\geq1$.
We further assume that the other zeros of $\mathcal{Q}_n(\lambda)$,
$n\geq1$, are distinct, that is $\mathcal{Q}_n(\lambda)$ has no repeated zeros.
We first find an eigenvalue representation
which will help in the study of a biorthogonality relation satisfied by
$\mathcal{Q}_n(\lambda)$. For this, we first note that
from the recurrence relations
\eqref{eqn: mixed ttrr for Q-2} and \eqref{eqn: mixed ttrr for for Q-n},
the leading coefficient of $\mathcal{Q}_n(\lambda)$ is
$\kappa_{n-1}=\rho_{n-1}\rho_{n-2}\cdots\rho_0\neq0$, $n\geq1$,
a fact that can also be verified from
\eqref{eqn: Q-n as linear combination in theorem}.

Further $\mathcal{Q}_n(\lambda)$, $n\geq1$, has the determinant representation
\begin{align}
\label{eqn: Q_n as a determinant}
\left|
\begin{array}{ccccc}
 \rho_0(\lambda-1) & \tau_0q_2q_1^{-1}(\lambda-1) & 0  & \cdots & 0 \\
 \lambda & q_1^{-1}(s_1\lambda+t_1) & -\tau_1q_3q_2^{-1} & \cdots & 0 \\
 0 & \lambda & \rho_2(\lambda-t_2r_2^{-1})  & \cdots & 0 \\
 0 & 0 & \lambda  & \cdots & 0 \\
 \vdots & \vdots & \vdots  & \ddots & \vdots \\
 0 & 0 & 0 & \cdots & \rho_{n-1}(\lambda-t_{n-1}r_{n-1}^{-1}) \\
 \end{array}
\right|.
\end{align}
\begin{theorem}
The zeros of the monic polynomials
$\mathcal{\hat{Q}}_n(\lambda)=\kappa_{n-1}^{-1}\mathcal{Q}_n(\lambda)$
are the eigenvalues of the matrix $\mathcal{D}_n$ given by
\begin{align*}
\left|
   \begin{array}{cccccc}
     \dfrac{s_1q_1^{-1}}{\rho_1}
& \dfrac{\tau_0\tau_1q_2q_1^{-1}}{\alpha_1\rho_1\rho_0}
& \dfrac{-\tau_0\tau_1q_3q_1^{-1}}{\rho_1\rho_0}
& 0
& \cdots
& 0 \\
     \dfrac{-1}{\rho_1}
& \dfrac{-\vartheta_2}{\alpha_1\rho_1}
& \dfrac{\tau_1q_3q_2^{-1}}{\rho_1}
& 0
& \cdots
& 0 \\
   \dfrac{1}{\rho_2\rho_1}
& \dfrac{\vartheta_2}{\alpha_1\rho_2\rho_1}
& \dfrac{-\vartheta_3}{\rho_2\rho_1}
& -\dfrac{\tau_2q_3^{-1}q_4}{\rho_2}
& \cdots
& 0 \\
     \dfrac{-1}{\rho_3\rho_2\rho_1}
& \dfrac{-\vartheta_2}{\alpha_1\rho_3\rho_2\rho_1}
& \dfrac{\vartheta_3}{\rho_3\rho_2\rho_1}
& \dfrac{-\vartheta_4}{\rho_3\rho_2}
& \cdots
& 0 \\
     \vdots & \vdots & \vdots  &\vdots&\ddots & \vdots \\
\dfrac{(-1)^{n+1}}{\rho_{n-1}\cdots\rho_1}
& \dfrac{(-1)^{n+1}\vartheta_2}{\alpha_1\rho_{n-1}\cdots\rho_1}
&\dfrac{(-1)^{n+2}\vartheta_3}{\rho_{n-1}\cdots\rho_1}
& \dfrac{(-1)^{n+3}\vartheta_4}{\rho_{n-1}\cdots\rho_2}
& \cdots
& \dfrac{-\vartheta_n}{\rho_{n-1}\rho_{n-2}} \\
   \end{array}
 \right|
\end{align*}
where $\vartheta_2=q_2$ and $\vartheta_n=
q_n(1+\rho_{n-2}\alpha_{n-1}^{-1})$,
$n\geq3$.
\end{theorem}
\begin{proof}
From \eqref{eqn: Q_n as a determinant},
$\mathcal{Q}_n(\lambda)=\lambda\mathcal{G}_n-\mathcal{H}_n$, where
\begin{align*}
\mathcal{G}_n=
\left|
  \begin{array}{ccccc}
    \rho_0 & \tau_0q_2q_1^{-1} & 0 & \cdots& 0 \\
    1 & s_1q_1^{-1} & 0 & \cdots & 0\\
    0 & 1 & \rho_2 & \cdots & 0\\
    \vdots & \vdots & \vdots & \ddots & \vdots\\
    0 & 0 & 0 & \cdots & \rho_{n-1}\\
  \end{array}
\right|
\qquad\mbox{and}
\end{align*}
\begin{align*}
\mathcal{H}_n=
\left|
  \begin{array}{ccccc}
    \rho_0 & \tau_0q_2q_1^{-1} & 0 & \cdots& 0 \\
    0 & -t_1q_1^{-1} & \tau_1q_3q_2^{-1} & \cdots & 0\\
    0 & 0 & -t_2q_2^{-1} & \cdots & 0\\
    \vdots & \vdots & \vdots & \ddots & \vdots\\
    0 & 0 & 0 & \cdots & -t_{n-1}q_{n-1}^{-1}\\
  \end{array}
\right|.
\end{align*}
Then, for $n\geq1$
\begin{align*}
\mathcal{Q}_n(\lambda)
=\det(\lambda\mathcal{G}_n-\mathcal{H}_n)
=\det(\mathcal{G}_n)\cdot\det(\lambda I-\mathcal{G}_n^{-1}\mathcal{H}_n).
\end{align*}
However $\det(\mathcal{G}_n)=\kappa_{n-1}\neq0$
implies that $\mathcal{G}_n$ is invertible and hence
$\mathcal{\hat{Q}}_n(\lambda)$
is the characteristic polynomial of the matrix product
$\mathcal{D}_n=\mathcal{G}_n^{-1}\mathcal{H}_n$.
The theorem now follows from the fact that the matrix inverse
$\mathcal{G}_n^{-1}$ is given by
\begin{align*}
\left|
  \begin{array}{ccccc}
    \dfrac{s_1q_1^{-1}}{\rho_1\rho_0} & -\dfrac{\tau_0q_2q_1^{-1}}{\rho_1\rho_0} & 0 & \cdots & 0 \\
    -\dfrac{1}{\rho_1\rho_0} & \dfrac{1}{\rho_1} & 0 & \cdots & 0 \\
    \dfrac{1}{\rho_2\rho_1\rho_0} & -\dfrac{1}{\rho_2\rho_1} & \dfrac{1}{\rho_2} & \cdots & 0 \\
    \vdots & \vdots & \vdots & \ddots & \vdots \\
    \dfrac{(-1)^{n+1}}{\rho_{n-1}\cdots\rho_0} & \dfrac{(-1)^{n+2}}{\rho_{n-1}\cdots\rho_1}
     & \dfrac{(-1)^{n+3}}{\rho_{n-1}\cdots\rho_2}  & \cdots & \dfrac{1}{\rho_{n-1}}  \\
  \end{array}
\right|
\end{align*}
and the proof is complete.
\end{proof}
Thus, with the assumptions made at the beginning of this section,
the eigenvalues of the matrix $\mathcal{D}_n$
are distinct.
Further, the intersection of the spectrum of $\mathcal{D}_n$ and
$\mathcal{D}_{n+1}$, $n\geq1$, is non-empty and
consists only of the point $\lambda=1$.

With $\sigma_0^{L}(\lambda):=\sigma_0^{R}(\lambda)
:=\mathcal{Q}_0(\lambda)=1$,
consider the following (Laurent) polynomials
\begin{align}
\label{eqn: definition of laurent polynomials for biorthogonality}
\sigma_k^{L}(\lambda):=(-\lambda)^{-k}\mathcal{Q}_k(\lambda),
\qquad
\sigma_k^{R}(\lambda):=-(\lambda-1)^{-1}\frac{\mathcal{Q}_k(\lambda)}
{\prod_{j=0}^{k-1}\tau_jq_{k+1}q_1^{-1}},
\end{align}
for $k=1,2,\cdots,n$.
Then, the mixed recurrence relations
\eqref{eqn: mixed ttrr for Q-2} and \eqref{eqn: mixed ttrr for for Q-n}
written in terms of $\sigma_{n}^{L}(\lambda)$
yield
\begin{align*}
\lambda[\rho_0\sigma_0^{L}(\lambda)+\sigma_1^{L}(\lambda)]
&=
\rho_0\sigma_0^{L}(\lambda)\\
\lambda[\tau_0q_2q_1^{-1}\sigma_0^{L}(\lambda)+
s_1q_1^{-1}\sigma_1^{L}(\lambda)+\sigma_2^{L}(\lambda)]
&=
\tau_0q_2q_1^{-1}\sigma_0^{L}(\lambda)-t_1q_1^{-1}\sigma_1^{L}(\lambda)\\
\lambda[\rho_n\sigma_n^{L}(\lambda)+\sigma_{n+1}^{L}(\lambda)]
&=
\tau_{n-1}q_{n+1}q_n^{-1}\sigma_{n-1}^{L}(\lambda)-t_nq_n^{-1}
\sigma_n^{L}(\lambda).
\end{align*}
Defining the sequence $\{\chi_n^{L}(\lambda)\}_{n=0}^{\infty}$
where
\begin{align}
\label{eqn: definition of chi-n-L}
\begin{split}
\chi_0^{L}(\lambda)
&=\rho_0\sigma_0^{L}(\lambda)+\sigma_1^{L}(\lambda),
\quad
\chi_1^{L}(\lambda)=\tau_0q_2q_1^{-1}\sigma_0^{L}(\lambda)+
s_1q_1^{-1}\sigma_1^{L}(\lambda)+\sigma_2^{L}(\lambda),\\
\chi_{n}^{L}(\lambda)
&=
\rho_{n}\sigma_{n}^{L}(\lambda)+\sigma_{n+1}^{L}(\lambda),
\quad
n\geq2,
\end{split}
\end{align}
we have
\begin{align}
\label{eqn: definition of chi-n-L with lambda}
\begin{split}
\lambda\chi_0^{L}(\lambda)
&=\rho_0\sigma_0^{L}(\lambda),
\quad
\lambda\chi_1^{L}(\lambda)=
\tau_0q_2q_1^{-1}\sigma_0^{L}(\lambda)-
t_1q_1^{-1}\sigma_1^{L}(\lambda)\\
\lambda\chi_{n}^{L}(\lambda)
&=
\tau_{n-1}q_{n+1}q_{n}^{-1}\sigma_{n-1}^{L}(\lambda)-
t_nq_{n}^{-1}\sigma_{n}^{L}(\lambda),
\quad n\geq2.
\end{split}
\end{align}
Similarly, the mixed recurrence relations
\eqref{eqn: mixed ttrr for Q-2} and
\eqref{eqn: mixed ttrr for for Q-n}
written in terms of $\sigma_{n}^{R}(\lambda)$
yield
\begin{align*}
\lambda[\rho_0\sigma_0^{R}(\lambda)+\tau_0q_2q_1^{-1}\sigma_1^{R}(\lambda)]
&=
\tau_0q_2q_1^{-1}\sigma_1^{R}(\lambda)+\rho_0\sigma_0^{R}(\lambda),\\
\lambda[\sigma_0^{R}(\lambda)+t_1q_1^{-1}\sigma_1^{R}(\lambda)]
&=
\tau_1q_3q_2^{-1}\sigma_2^{R}(\lambda)-t_1q_1^{-1}\sigma_1^{R}(\lambda),\\
\lambda[\sigma_{n-1}^{R}(\lambda)+\rho_n\sigma_{n}^{R}(\lambda)]
&=
\tau_nq_{n+2}q_{n+1}^{-1}\sigma_{n+1}^{R}(\lambda)-t_{n}q_{n}^{-1}
\sigma_{n}^{R}(\lambda).
\end{align*}
Defining the sequence $\{\chi_n^{R}(\lambda)\}_{n=0}^{\infty}$ where
\begin{align}
\label{eqn: definition chi-n-R}
\begin{split}
\chi_0^{R}(\lambda)&=\rho_0\sigma_0^{R}(\lambda)+
\tau_0q_2q_1^{-1}\sigma_1^{R}(\lambda),
\quad
\chi_1^{R}(\lambda)=\sigma_0^{R}(\lambda)+t_1q_1^{-1}\sigma_1^{R}(\lambda),\\
\chi_{n}^{R}(\lambda)
&=
\sigma_{n-1}^{R}(\lambda)+\rho_{n}\sigma_{n}^{R}(\lambda),
n\geq2.
\end{split}
\end{align}
we have
\begin{align}
\label{eqn: definition of chi-n-R with lambda}
\begin{split}
\lambda\chi_0^{R}(\lambda)
&=
\tau_0q_2q_1^{-1}\sigma_1^{R}(\lambda)+\rho_0\sigma_0^{R}(\lambda),
\quad
\lambda\chi_1^{R}(\lambda)
=
\tau_1q_3q_2^{-1}\sigma_2^{R}(\lambda)-
t_1q_1^{-1}\sigma_1^{R}(\lambda),\\
\lambda\chi_{n}^{R}(\lambda)
&=
\tau_nq_{n+2}q_{n+1}^{-1}\sigma_{n+1}^{R}(\lambda)-
t_nq_n^{-1}\sigma_n^{R}(\lambda),
\quad n\geq2.
\end{split}
\end{align}
The matrix equations corresponding to
\eqref{eqn: definition of chi-n-L with lambda} and
\eqref{eqn: definition of chi-n-R with lambda}
 are respectively
\begin{subequations}
\begin{align}
\label{eqn: matrix equation for sigma-L}
\lambda\vec{\sigma}^{L}(\lambda)\mathcal{G}_n
&=
\vec{\sigma}^{L}(\lambda)\mathcal{H}_n+
\lambda\sigma_n^{L}(\lambda)\vec{e}_n^{T}
\quad\mbox{and}\\
\lambda\mathcal{G}_n\vec{\sigma}^{R}(\lambda)
&=\mathcal{H}_n\vec{\sigma}^{R}(\lambda)+
\tau_{n-1}q_{n+1}q_n^{-1}\sigma_n^{R}(\lambda)\vec{e}_n,
\label{eqn: matrix equation for sigma-R}
\end{align}
\end{subequations}
where $\vec{e}_n$ is the $n^{th}$ column of the $n\times n$
identity matrix $I_n$ and
\begin{align*}
\vec{\sigma}^{L}(\lambda)
=
[\sigma_0^{L}(\lambda)\,\,\sigma_1^{L}(\lambda)\cdots\sigma_{n-1}^{L}(\lambda)]
\quad\mbox{and}\quad
\vec{\sigma}^{R}(\lambda)
=
[\sigma_0^{R}(\lambda)\,\,\sigma_1^{R}(\lambda)
\cdots\sigma_{n-1}^{R}(\lambda)]^{T}.
\end{align*}
\begin{lemma}
\label{lemma: lemma for biorthogonality of finite sequences}
Let the zeros of $\mathcal{Q}_n(\lambda)$ be denoted as
$\lambda_{n,j}$, $j=1,2,\cdots,n$, with $\lambda_{n,n}=1$.
Then, the following biorthogonality relation
\begin{align}
\label{eqn: biorthogonality relation in theorem}
\sum_{i=0}^{n-1}\sigma_i^{R}(\lambda_{n,j})
\tilde{\chi}_i^{L}(\lambda_{n,k})\mu_{n,j,k}
=\delta_{j,k},
\quad
j,k=0,1,\cdots,n-2,
\end{align}
holds, where
$\tilde{\chi}_i^{L}(\lambda_{n,k})=\chi_i^{L}(\lambda_{n,k})$,
$i=1,2,\cdots,n-2$ and
$\tilde{\chi}_{n-1}(\lambda)=\rho_{n-1}\sigma_{n-1}^{L}(\lambda)$.
Further, the weight function $\mu_{n,j,k}$ has the expression
\begin{align*}
\mu_{n,j,k}=[\tau_{n-1}q_{n+1}q_n^{-1}[\sigma_{n}^{R}(\lambda_{n,j})]^{'}
\sigma_{n-1}^{L}(\lambda_{n,k})]^{-1},
\quad
j,k=1,\cdots,n-1.
\end{align*}
\end{lemma}
\begin{proof}
Post-multiplying \eqref{eqn: matrix equation for sigma-R}
by $\vec{\sigma}^{L}(\lambda)$
and pre-multiplying \eqref{eqn: matrix equation for sigma-L}
by $\vec{\sigma}^{R}(\omega)$
after evaluating \eqref{eqn: matrix equation for sigma-R}
at $\omega$, we obtain the systems
\begin{subequations}
\begin{align}
\label{eqn: matrix equation for sigma-L after evaluation at omega}
\lambda\vec{\sigma}^{R}(\omega)\vec{\sigma}^{L}(\lambda)\mathcal{G}_n
&=
\vec{\sigma}^{R}(\omega)\vec{\sigma}^{L}(\lambda)\mathcal{H}_n+
\lambda\vec{\sigma}^{R}(\omega)\sigma_n^{L}(\lambda)\vec{e}_n^{T}\\
\omega\mathcal{G}_n\vec{\sigma}^{R}(\omega)\vec{\sigma}^{L}(\lambda)
&=\mathcal{H}_n\vec{\sigma}^{R}(\omega)\vec{\sigma}^{L}(\lambda)+
\tau_{n-1}q_{n+1}q_n^{-1}\sigma_n^{R}(\omega)\vec{e}_n\vec{\sigma}^{L}(\lambda).
\label{eqn: matrix equation for sigma-R after evaluation at omega}
\end{align}
\end{subequations}
It can be verified that the trace of
$\vec{\sigma}^{R}(\omega)\vec{\sigma}^{L}(\lambda)\mathcal{G}_n$
which is also equal to the trace of
$\mathcal{G}_n\vec{\sigma}^{R}(\omega)\vec{\sigma}^{L}(\lambda)$
is same as the matrix product
$\vec{\sigma}^{L}(\lambda)\mathcal{G}_n\vec{\sigma}^R(\omega)$.
Further, using the linearity property of the trace
and by subtracting the trace of
\eqref{eqn: matrix equation for sigma-R after evaluation at omega}
from the trace of
\eqref{eqn: matrix equation for sigma-L after evaluation at omega},
we obtain
\begin{align}
\label{eqn: Christoffel formula from matrix equations}
\vec{\sigma}^{L}(\lambda)\mathcal{G}_n\vec{\sigma}^{R}(\omega)=
\frac{\lambda\sigma_{n-1}^{R}(\omega)\sigma_{n}^{L}(\lambda)-
\tau_{n-1}q_{n+1}q_n^{-1}\sigma_{n}^{R}(\omega)\sigma_{n-1}^{L}(\lambda)}
{\lambda-\omega}.
\end{align}
Further, excluding the point $\lambda_{n,n}=\omega_{n,n}=1$ and using
$\omega\rightarrow\lambda$ in
\eqref{eqn: Christoffel formula from matrix equations}
we have
\begin{align*}
\vec{\sigma}^{L}(\lambda)\mathcal{G}_n\vec{\sigma}^{R}(\lambda)=
\tau_{n-1}q_{n+1}q_n^{-1}[\sigma_{n}^{R}(\lambda)]^{'}
\sigma_{n-1}^{L}(\lambda)-
\lambda[\sigma_{n-1}^{R}(\lambda)]^{'}\sigma_{n}^{L}(\lambda).
\end{align*}
Hence at the zeros $\lambda_{n,j}$, $j=1,2,\cdots,n-1,$
\begin{align*}
\vec{\sigma}^{L}(\lambda_{n,j})\mathcal{G}_n\vec{\sigma}^{R}(\lambda_{n,j})=
\tau_{n-1}q_{n+1}q_n^{-1}[\sigma_{n}^{R}(\lambda_{n,j})]^{'}
\sigma_{n-1}^{L}(\lambda_{n,j}).
\end{align*}
Define
$\mu_{n,j}^{-1}:=\tau_{n-1}q_{n+1}q_n^{-1}[\sigma_{n}^{R}(\lambda_{n,j})]^{'}
\sigma_{n-1}^{L}(\lambda_{n,j})$.
Since the zeros are assumed to be distinct
and Theorem
\ref{thm: theorem for Qn and Qn-1 not having common zero}
gives
$\mathcal{Q}_{n-1}(\lambda_{n,j})\neq0$,
we have
$\mu_{n,j}^{-1}\neq0$, $j=1,\cdots,n-1$.
Hence from \eqref{eqn: Christoffel formula from matrix equations},
for $j,k=1,2,\cdots,n-1,$
\begin{align}
\label{eqn: biorthogonality relation from Gn matrix and sigma L-R}
\left[\mathcal{G}_n^{T}[\vec{\sigma}^{L}(\lambda_{n,j})]^T\right]^{T}
\vec{\sigma}^{R}(\lambda_{n,k})=
[\vec{\sigma}^{R}(\lambda_{n,k})]^{T}
\mathcal{G}_n^{T}[\vec{\sigma}^{L}(\lambda_{n,j})]^{T}
=
\mu_{n,j}^{-1}\delta_{j,k}.
\end{align}
This shows the two finite sequences
$\{\mathcal{G}_n^{T}[\vec{\sigma}^{L}(\lambda_{n,j})]^T\}_{j=1}^{n-1}$ and
$\{\vec{\sigma}^{R}(\lambda_{n,k})\}_{k=1}^{n-1}$ are biorthogonal to each other.

To proceed further in the proof, we define
$\mu_{n,j,k}^{-1}:=\tau_{n-1}q_{n+1}q_n^{-1}[\sigma_{n}^{R}(\lambda_{n,j})]^{'}
\sigma_{n-1}^{L}(\lambda_{n,k})$, where
we note that for $k=j$,
$\mu_{n,j,j}^{-1}=\mu_{n,j}^{-1}$.
We further define two non-singular matrices
$\mathcal{A}_{n\times n-1}=(a_{m,j})$ and
$\mathcal{B}_{n-1\times n}=(b_{j,m})$
where for $j=1,2,\cdots, n-1$ and $m=1,2,\cdots, n$
\begin{align*}
a_{m,j}=\frac{\sigma_{m-1}^{R}(\lambda_{n,j})}
{\tau_{n-1}q_{n+1}q_n^{-1}[\sigma_n^{R}(\lambda_{n,j})]^{'}},
\quad
b_{j,m}=\frac{\tilde{\chi}_{m-1}(\lambda_{n,j})}{\sigma_{n-1}^{L}(\lambda_{n,j})}.
\end{align*}
It follows from
\eqref{eqn: biorthogonality relation from Gn matrix and sigma L-R}
that the matrix relation
$\mathcal{B}_{n-1\times n}\cdot\mathcal{A}_{n\times n-1}=I_{n-1}$ holds.
The system of equations resulting from
$\mathcal{A}^{T}_{n-1\times n}\cdot\mathcal{B}^{T}_{n-1\times n}=I_{n-1}$
can be written as
\begin{align*}
\sum_{i=0}^{n-1}\sigma_i^{R}(\lambda_{n,j})\tilde{\chi}_i^{L}(\lambda_{n,k})\mu_{n,j,k}
=\delta_{j,k},
\quad
j,k=1,\cdots,n-1,
\end{align*}
thus proving that the two finite sequences
$\{\sigma_i^{R}(\lambda)\}_{i=0}^{n-2}$
and
$\{\tilde{\chi}_k^{L}(\lambda)\}_{k=0}^{n-2}$
are biorthogonal to each other
on the point set $\lambda=\lambda_{n,1},\lambda_{n,2},\cdots,\lambda_{n,n-1}$,
which is the set of zeros of $\mathcal{Q}_n(\lambda)$ excluding the point $\lambda=1$.
\end{proof}

If $\bigwedge$ denotes the space of Laurent
polynomials,
define a linear functional
$\mathcal{N}$ on $\bigwedge\times\bigwedge$
as
\begin{align*}
\mathcal{N}_{n-1}^{(k,j)}[h_i(\lambda)\,g_i(\lambda)]=
\sum_{i=0}^{n-1}h_{i}(\lambda_{n,j})g_{i}(\lambda_{n,k})\mu_{n,j,k}.
\end{align*}
Then, we have the following result.
\begin{theorem}
\label{thm: theorem for orthogonality of Qn}
The sequence $\{\mathcal{Q}_n(\lambda)\}_{n=1}^{\infty}$ satisfy
\begin{align*}
\mathcal{N}_{n-1}^{(k,j)}
\left[\lambda_{n,k}^{-n+m}(\lambda-1)^{-1}\mathcal{Q}_i(\lambda)\right]
=0,
\quad
\mathcal{N}_{n-1}^{(k,k)}\left[\lambda_{n,k}^{-n+m}(\lambda-1)^{-1}\mathcal{Q}_i(\lambda)\right]
\neq0,
\end{align*}
for $k=1,\cdots,n-1$,
and $m=1,2,\cdots,n$.
\end{theorem}
\begin{proof}
From Lemma~\ref{lemma: lemma for biorthogonality of finite sequences}
we have
\begin{align*}
\mathcal{N}_{n-1}^{(k,j)}[\sigma_i^{R}(\lambda)\tilde{\chi}_i(\lambda)]=\delta_{j,k},
\quad 0\leq i\leq n-1,
\quad 1\leq j,k\leq n-1.
\end{align*}
Using the definitions \eqref{eqn: definition of chi-n-L}
of $\chi_k^{L}(\lambda)$,
it is clear that
\begin{align*}
\{\tilde{\chi}_0^{L}(\lambda_{n,k}),\tilde{\chi}_1^{L}(\lambda_{n,k}),\cdots,
\tilde{\chi}_{n-1}^{L}(\lambda_{n,k})\},
\end{align*}
(where $\tilde{\chi}_j^{L}(\lambda)=\chi_j^{L}(\lambda)$,
$j=0\cdots,n-2$ and
$\tilde{\chi}_{n=1}(\lambda)=
\rho_{n-1}\sigma_{n-1}^{L}(\lambda)$),
forms a basis for the subspace of Laurent polynomials spanned
by $\{1,\lambda_{n,k}^{-1},\lambda_{n,k}^{-2},\cdots,
\lambda_{n,k}^{-n+1}\}$.
Note that each fixed $k$ yields $n-1$ such subspaces.
Further, this implies that
\begin{align*}
\mathcal{N}_{n-1}^{(k,j)}[\lambda_{n,k}^{-n+m}\sigma_i^{R}(\lambda)]=0,
\quad
j,k=0,1,\cdots,n-1,
\quad
m=1,2,\cdots,n,
\end{align*}
while
$\mathcal{N}_{n-1}^{(k,j)}[\lambda_{n,k}^{-n+m}\sigma_i^{R}(\lambda)]\neq0$,
whenever $\sigma_i^{R}(\lambda)$ is evaluated at $\lambda=\lambda_{n,k}$.
The theorem now follows from the fact that
$\sigma_i^{R}(\lambda)=
-[\prod_{j=0}^{i-1}\tau_jq_{i+1}q_{1}^{-1}(\lambda-1)]^{-1}
\mathcal{Q}_i(\lambda)$.
\end{proof}

Biorthogonality of polynomial sequences as well as their
derivatives is well studied. We refer the reader to
\cite{Askey-discussion-Szego-paper-1982,
Askey-problems-on-SF-computations-1985,
Hendriksen-Njastad-biorthogonal-derivative-Rocky-1991,
Temme-biorthogonal-Constapprx-1986}
for examples of such sequences and related discussions.
For example, biorthogonality relations satisfied only up to a finite
number of polynomials in a sequence
have been obtained in
\cite{Silva-Ranga-bounds-complex-zeros-JAT-2005,
Zhedanov-ortho-polygons-JAT-1999}.
The crucial point in such analysis is the use of
the Christoffel-Darboux type formula
\eqref{eqn: Christoffel formula from matrix equations}
which necessitates the
removal of the common zeros of polynomials of consecutive degrees.

\section{Para-orthogonality relations}
\label{sec: para-orthogonality relations}
The orthogonality condition in
Theorem \ref{thm: theorem for orthogonality of Qn}
motivates us to study the sequence of polynomials in which the
common zero $\lambda=1$ has been removed.
Thus, we consider the sequence
$\{\mathcal{R}_n(\lambda)\}_{n=1}^{\infty}$ where
$
\mathcal{R}_n(\lambda)=\kappa_n^{-1}(\lambda-1)^{-1}
\mathcal{Q}_{n+1},
n\geq0.
$
We recall that $\kappa_n$ is the leading coefficient of
$\mathcal{Q}_{n+1}(\lambda)$,
thus making $\mathcal{R}_n(\lambda)$ monic.

For para-orthogonality, we impose conditions on the parameters used in
\eqref{eqn: ttrr for RI polynomials Pn}. First, $-2<\tau_n\rho_n^{-1}<0$,
$n\geq0$.
By a direct computation from \eqref{eqn: mixed ttrr for Q-2},
we obtain $\mathcal{Q}_2(\lambda)=
\kappa_2(\lambda-1)(\lambda+t_1\rho_1^{-1}q_1^{-1})$.
Then, $\mathcal{Q}_2(\lambda)$ has the second zero at
$\lambda=-1$ if $t_1=\rho_1q_1$
which is equivalent to
 $\alpha_1=\tau_1\rho_1^{-1}\beta_0(1+\beta_0)^{-1}$.
 Since $\alpha_1=\rho_0(\beta_0-1)$ this narrows down the
 choice of the initial values to satisfy
 $\tau_1\rho_1^{-1}=\rho_0\beta_0^{-1}(\beta_0^{2}-1)$.
We note that such a choice does not
contradict the inequalities at the end of
Section \ref{sec: Section-2-mixed recurrence relation}
since $\beta_0\neq\pm1$.

Hence from \eqref{eqn: mixed ttrr for for Q-n},
$\mathcal{R}_n(\lambda)$
satisfies the recurrence relation
\begin{align}
\label{eqn: ttrr for Rn from ttrr for Qn after dividing by lambda-1}
\mathcal{R}_{n+1}(\lambda)=
(\lambda+t_{n+1}\rho_{n+1}^{-1}q_{n+1}^{-1})\mathcal{R}_{n}(\lambda)+
\frac{\tau_np_{n+2}}{\rho_nr_{n+1}}\lambda\mathcal{R}_{n-1}(\lambda),
\quad n\geq1,
\end{align}
with the initial conditions $\mathcal{R}_{0}(\lambda)=1$ and
$\mathcal{R}_{1}(\lambda)=\lambda+1$.
\begin{theorem}
Let the sequence $\{\alpha_n\}_{n=1}^{\infty}$ be so constructed
from \eqref{eqn: condition for mixed recurrence relation}
that, the following conditions are also satisfied
\begin{subequations}
\begin{align}
\label{eqn: condition1 on alpha-n for POP}
\alpha_n\rho_n
&=
(2\rho_n+\tau_n)(\rho_{n-1}\alpha_{n-1}-\tau_{n-1})+\tau_n,
\quad\mbox{and}\\
\alpha_{n}\rho_n
&=
-\rho_{n-1}\beta_{n-1}(2\rho_n+\tau_n)\alpha_{n-1},
n\geq2,
\label{eqn: condition2 on alpha-n for POP}
\end{align}
\end{subequations}
where $\alpha_0=\rho_0\tau_0^{-1}$ and
$\alpha_1=\rho_0(\beta_0-1)$.
Then $\{\mathcal{R}_n(\lambda)\}_{n=1}^{\infty}$
is a sequence of para-orthogonal polynomials.
\end{theorem}
\begin{proof}
The conditions \eqref{eqn: condition1 on alpha-n for POP} and
\eqref{eqn: condition2 on alpha-n for POP} imply the relations
$p_{n+1}=(2\rho_n+\tau_n)p_n$ and $t_n=-r_n$, $n\geq2$. Further
\begin{align*}
\frac{p_{n+1}}{p_n}=2\rho_n+\tau_n
\Longrightarrow
\frac{p_{n+1}}{r_n}-\frac{\tau_n}{\rho_n}=2,
\quad
n\geq2.
\end{align*}
Hence, choosing $\tau_n\rho_n^{-1}=-2(1-m_n)$, so that $0<m_n<1$
(since $-2<\tau_n\rho_n^{-1}<0$), $n\geq1$
and
$p_{n+1}r_{n}^{-1}=2m_n$, $n\geq2$, the recurrence relation
\eqref{eqn: ttrr for Rn from ttrr for Qn after dividing by lambda-1}
for $\mathcal{R}_n$ reduces to
\begin{align}
\label{eqn: ttrr for Rn as POP with chain sequence}
\mathcal{R}_{n+1}(\lambda)=(\lambda+1)\mathcal{R}_n(\lambda)-
4d_{n+1}\lambda\mathcal{R}_{n-1}(\lambda),
\quad
n\geq1,
\end{align}
with $\mathcal{R}_0(\lambda)=1$ and
$\mathcal{R}_1(\lambda)=\lambda+1$.
Since $d_{n+1}=(1-m_n)m_{n+1}$, $n\geq1$,
is a positive chain sequence
(see \cite[Section 7.2]{Ismail-book} for the definition of chain sequence),
by the results obtained in
\cite{Ranga-Favard-type-OPUC-JAT-2014},
 it follows that
$\{\mathcal{R}_n(\lambda)\}_{n=1}^{\infty}$
is a sequence of para-orthogonal polynomials.
\end{proof}
The three term recurrence relation
\eqref{eqn: ttrr for Rn as POP with chain sequence}
has also been obtained
\cite[eq.~17]{Delsarte-Genin-split-Levinson-IEEE-1986}
in the study of Szeg\H{o} polynomials, wherein
the polynomials $\mathcal{R}_n(\lambda)$
are called the singular predictor polynomials with the notation
$``\lambda_{k}"$ 
equal to $-\tau_{n+1}\rho_{n+1}^{-1}$.
However we note that such three term recurrence relations
have been generalized to include complex parameters,
where the polynomials $\mathcal{R}_n(\lambda)$
are called Delsarte and Genin~$1$~para-orthogonal polynomials (DG1POP).
For this generalized theory and related discussion,
we refer to \cite{Ranga-Swami-POP-OPUC-ANM-2016,
Ranga-OPUC-chain-sequence-JAT-2013}
and references therein.

Further from
\cite[Theorem 4.1]{Ranga-Favard-type-OPUC-JAT-2014} there exists a
non-trivial probability measure on the unit circle such that for $n\geq1,$
\begin{align*}
\int_{\mathcal{C}}\zeta^{-n+k}\mathcal{R}_n(\zeta)(1-\zeta)d\mu(\zeta)=0,
\quad
k=0,1,\cdots,n-1,
\end{align*}
so that the corresponding polynomials $\mathcal{Q}_n(\lambda)$, $n\geq1$,
satisfy
\begin{align*}
\int_{\mathcal{C}}\zeta^{-n+k}\mathcal{Q}_n(\zeta)d\mu(\zeta)=0,
\quad
k=0,1,\cdots,n-1.
\end{align*}
Moreover, the monic polynomials
\begin{align*}
\phi_n(\lambda)=\mathcal{R}_n(\lambda)-2(1-m_n)\mathcal{R}_{n-1}(\lambda)
=
\mathcal{R}_n(\lambda)+\tau_n\rho_n^{-1}\mathcal{R}_{n-1}(\lambda),
\quad n\geq1,
\end{align*}
are orthogonal polynomials on the unit circle
 with respect to the probability measure $\mu$
and $\alpha_{n-1}=-\overline{\phi_n(0)}=
-(1+\tau_n\rho_n^{-1}), $ $n\geq1$,
\cite[Theorem 5.2]{Ranga-Favard-type-OPUC-JAT-2014}.
The parameters $\alpha_{n-1}$ are called Verblunsky coefficients
\cite{Simon-book-Part-1} and (since $-2<\tau_n\rho_n^{-1}<0$)
lie in the real interval $[-1,1]$ in the present case.
\begin{remark}
The polynomials $\phi_n(\lambda)$, $n\geq1$,
are a well-known class of $R_{I}$
polynomials that can be expressed as a linear combination
of consecutive polynomials of another class of $R_{I}$ polynomials
with the choice
$\alpha_n=\tau_n\rho_{n}^{-1}$, $n\geq1$, thus verifying
Corollary \ref{coro: corollary for choice of alpha-n for R-I polynomials}.
\end{remark}
\begin{remark}
By the results obtained in previous sections, a para-orthogonal polynomial
can always be obtained from a linear combination of $R_{I}$ polynomials
satisfying appropriate three term recurrence relations.
In fact, expressing the parameters
$\rho_n$, $\beta_n$ and $\tau_n$ in terms of $\alpha_n$
can lead to the identification of the class of $R_I$
polynomials whose linear combinations with constant
coefficients lead to para-orthogonal polynomials.
\end{remark}

\section{A hypergeometric case}
\label{sec: Section-3-illustration}
 Consider the recurrence relation of $R_{I}$ type
\begin{align*}
\mathcal{P}_{n+1}(\lambda)=
\frac{b+n}{c+n}\left(\lambda-\frac{b-c-n}{b+n}\right)\mathcal{P}_{n}(\lambda)-
\frac{n}{c+n}\lambda\mathcal{P}_{n-1}(\lambda),
\quad
n\geq1,
\end{align*}
where $\mathcal{P}_{0}(\lambda)=1$ and
$\mathcal{P}_{1}(\lambda)=\frac{b}{c}(\lambda-\frac{b-c}{b})$.
Hence we have the parameters as
\begin{align*}
\rho_{n}=\frac{b+n}{c+n},\quad
\beta_n=\frac{b-c-n}{b+n},
\quad
\tau_{n+1}=-\frac{n+1}{c+n+1},
\quad \gamma_n=0,
\quad n\geq0.
\end{align*}
Further,
$\mathcal{P}_n(\lambda)=F(-n,b;c;1-\lambda)$, $n\geq 0$,
which follows easily from the contiguous relation
\cite{Ranga-szego-polynomials-2010-AMS}
\begin{align*}
(a-c+1)F(a,b;c;\lambda)=
(2a-c&+2+(b-a-1)\lambda)F(a+1,b;c;\lambda)\\
&+(a+1)(\lambda-1)F(a+2,b;c;\lambda).
 \end{align*}
We consider $\mathcal{Q}_n(\lambda)=
\mathcal{P}_n(\lambda)+\alpha_n\mathcal{P}_{n-1}(\lambda)$, $n\geq0$,
as given in \eqref{eqn: Q-n as linear combination in theorem}
and our first aim is to construct the sequence
$\{\alpha_n\}_{n=0}^{\infty}$ so that
$\{\mathcal{Q}_n(\lambda)\}_{n=1}^{\infty}$
has a common zero at $\lambda=1$.
We start with certain initial verifications given at the end of Section
\ref{sec: Section-2-mixed recurrence relation}.

Clearly, $c\neq0$ in $F(-n,b;c;1-\lambda)$ and so $\beta_0\neq1$.
Since we also require $\beta_0\neq0,$ and $\beta_0\neq,-1$,
we exclude the relations $c=b$ and $c=2b$ respectively.
We further verify that $-(b+1)^{-1}=\tau_1\rho_1^{-1}\neq
\rho_0\beta_0^{-1}(1-\beta_0)^2=c(b-c)^{-1}$ as $c\neq-1$,
which by \eqref{eqn: condition for double zero of Q2 at 1}
implies $\mathcal{Q}_2(\lambda)$ does not have a double root at $\lambda=1$.
We now choose
\begin{align*}
\alpha_1=\rho_0(\beta_0-1)=\frac{b}{c}\left(\frac{b-c}{b}-1\right)=-1,
\end{align*}
so that from \eqref{eqn: condition for mixed recurrence relation},
$\mathcal{Q}_n(\lambda)$ has a common zero at $\lambda=1$, $n\geq1$, for
the unique sequence
$\alpha_n=-1$, $n\geq1$.
Moreover, as $b\neq0$ and $c\neq2b$, $-n(b+n)^{-1}=\tau_n\rho_n^{-1}\neq \alpha_n=-1$,
and $(b-c-n+1)(c+n-1)^{-1}=\rho_{n-1}\beta_{n-1}\neq-1$, $n\geq1$, respectively.
Thus, the linear combinations
\begin{align*}
\mathcal{Q}_n(\lambda)=F(-n,b;c;1-\lambda)-F(-n+1,b;c;1-\lambda),
\quad
n\geq1,
\end{align*}
satisfy the mixed recurrence relations
\eqref{eqn: mixed ttrr for Q-2} and \eqref{eqn: mixed ttrr for for Q-n}.
With $\alpha_{n+1}=-1$, $\gamma_n=0$, $n\geq0$, the relations in
Theorem $\ref{thm: ttrr for Q-n}$
yield
\begin{align*}
-p_n&=q_n=\frac{b}{c+n-1},\,\,
r_n=-\frac{b(b+n)}{(c+n)(c+n-1)},\,\,
s_n=\frac{b(2b-c+1)}{(c+n-1)(c+n)},\\
t_n&=-\frac{b(b-c-n+1)}{(c+n-1)(c+n)},\,\,
u_n=-v_n=\frac{(n-1)b}{(c+n-1)(c+n)},\,\,
w_n=0,
\quad n\geq2.
\end{align*}
For $n=1$, we have
\begin{align*}
p_1&=-\left(\frac{b}{c}+\tau_0\right),\,\,
q_1=\frac{b}{c},\,\,
r_1=-\frac{b+1}{c+1}\left(\frac{b}{c}+\tau_0\right),\,\,
s_1=\frac{\tau_0b}{c+1}+\frac{b(2b-c+1)}{c(c+1)},\nonumber\\
t_1&=-\frac{b(b-c)}{c(c+1)},\,\,
-u_1=v_1=\frac{\tau_0 b}{c+1},\,\,
w_1=0 .
\end{align*}
Hence we need to choose $\tau_0=-b/c$ so that $p_1=r_1=0$.
Then, from
\eqref{eqn: mixed ttrr for Q-2} and
\eqref{eqn: mixed ttrr for for Q-n}, we obtain
the mixed recurrence relations
\begin{align*}
\mathcal{Q}_2(\lambda)
&=
\frac{b-c+1}{c+1}\left(\lambda-\frac{b-c}{b-c+1}\right)
\mathcal{Q}_1(\lambda)+\frac{b}{c+1}\lambda(\lambda-1)
\mathcal{Q}_0(\lambda),\nonumber\\
\mathcal{Q}_{n+1}(\lambda)
&=
\frac{b+n}{c+n}\left(\lambda-\frac{b-c-n+1}{b+n}\right)
\mathcal{Q}_{n}(\lambda)-
\frac{n-1}{c+n}\lambda\mathcal{Q}_{n-1}(\lambda),
\quad n\geq 2.
\end{align*}
The initial conditions are $\mathcal{Q}_0(\lambda)=1$ and
$\mathcal{Q}_1(\lambda)=\frac{b}{c}(\lambda-1)$.
Note that $\rho_0=\frac{b}{c}$ gives $\alpha_0=-1$.
Moreover, from the power series representations
for hypergeometric functions, it
can be easily proved that
\begin{align*}
\mathcal{Q}_0(\lambda)=1,\quad \mathcal{Q}_n(\lambda)=
\frac{b}{c}(\lambda-1)F(-n+1,b+1;c+1;1-\lambda),
\quad n\geq1.
\end{align*}
To discuss the biorthogonality relations for
$\mathcal{Q}_n(\lambda)$,
we first note from
\eqref{eqn: definition of laurent polynomials for biorthogonality}
that
\begin{align*}
\sigma_i^{R}(\lambda)=
\frac{c+i}{c}\frac{(c+1)_{i-1}}{(-i+1)_{i-1}}
F(-i+1,b+1;c+1;1-\lambda),
\quad
i\geq2,
\end{align*}
with $\sigma_0^{R}(\lambda)=1$ and
$\sigma_1^{R}(\lambda)=\frac{c+1}{c}$.
Further, for $k\geq3$,
\begin{align*}
\chi_k^{L}(\lambda)=(-1)^{k+1}\frac{b}{c}
\frac{(\lambda-1)}{\lambda^{k+1}}
[F(-k,&b+1;c+1;1-\lambda)\\
&-\frac{b+k}{c+k}\lambda F(-k+1,b+1;c+1;1-\lambda)].
\end{align*}
Using another contiguous relation
\begin{align*}
(a-c)F(a-1,b;c,\lambda)+(c-b)F(a,b-1;c;\lambda)
+(\lambda-1)(a-b)F(a,b;c;\lambda)=0,
\end{align*}
it can be shown that
\begin{align*}
\chi_{k}^{L}(\lambda)=(-1)^{k+1}\frac{b(c-b)}{c(c+k)}
\frac{\lambda-1}{\lambda^{k+1}}
F(-k+1,b;c+1;1-\lambda),
\quad k\geq2.
\end{align*}
We also have $\chi_{0}^{L}(\lambda)=
\frac{b}{c\lambda}$ and
\begin{align*}
\chi_{1}^{L}(\lambda)=
\frac{b}{c}\frac{\lambda-1}{\lambda^2}
F(-1,b+1;c+1;1-\lambda)-
\frac{b(b+1)}{c(c+1)}\frac{\lambda-1}{\lambda}-
\frac{b}{(c+1)\lambda}.
\end{align*}
Note that we have written $\chi_{1}^{L}(\lambda)$ in the form
$\chi_{1}^{L}(\lambda)=
\sigma_2^{L}(\lambda)+\rho_1\sigma_1^{L}(\lambda)-
\frac{b}{(c+1)\lambda}$
and this can be further simplified to
$\chi_{1}^{L}(\lambda)=
\frac{-b}{(c+1)\lambda^2}F(-1,b;c;1-\lambda)$.

Let the weight function $\mu_{n,j,k}$ be the function
$\mu_{\lambda}$ evaluated at the zeros $\lambda_{n,j}$, $j=1,\cdots, n-1$ of
$\mathcal{Q}_n(\lambda)$, that is, at the $n-1$ zeros of
$F(-n+1,b+1;c+1;1-\lambda)$
which clearly does not include the point
$\lambda=1$.
Now, for $n\geq2$, we have
\begin{align*}
\mu_{\lambda}=\mathcal{Y}_{n-2}(\lambda)
[F(-n+2,b+1;c+1;1-\lambda)\times
F(-n+2,b+2;c+2;1-\lambda)]^{-1},
\end{align*}
where
\begin{align*}
\mathcal{Y}_{n-2}(\lambda)=\frac{c^2(-n+2)_{n-2}\lambda^{n-1}}
{(-1)^{n-1}(n-1)b(b+1)(c+2)_{n-2}(\lambda-1)},
\quad
n\geq2.
\end{align*}
By Chu-Vandermonde formula
\cite[p.12, (1.4.3)]{Ismail-book}
$F(-n+1,b+1;c+1;1)=\frac{(c-b)_{n-1}}{(c+1)_{n-1}}\neq0$.
This implies
$\lambda_{n,j}\neq0$, $j=1,\cdots n-1$.

Hence we obtain two finite sequences
$\{\sigma_i^{R}(\lambda)\}_{i=0}^{n-2}$
and
$\{\chi_k^{L}(\lambda)\}_{k=0}^{n-2}$
of hypergeometric functions
that are biorthogonal to each other with the weight
function $\mu_{\lambda}$, all quantities being evaluated
at $\lambda=\lambda_{n,j}$, $j=1,\cdots,n-1$, $n\geq2$.

Now, we discuss the para-orthogonality of polynomials
obtained from $\mathcal{Q}_n(\lambda)$. Since the
leading coefficient of $\mathcal{Q}_{n+1}(\lambda)$ is
$\kappa_n=\frac{(b)_{n+1}}{(c)_{n+1}}$, we
consider the polynomials
\begin{align*}
\mathcal{R}_n(\lambda)=
\frac{(c+1)_{n}}{(b+1)_{n}}F(-n,b+1,c+1,1-\lambda),
\quad
n\geq0.
\end{align*}
The other conditions obtained in the beginning of
Section \ref{sec: para-orthogonality relations} are
$-2<\tau_n\rho_n^{-1}<0$ and
$\tau_1\rho_1^{-1}=\rho_0\beta_0^{-1}(\beta_0^2-1)$.
The first one requires $b>-n/2$ for $n\geq1$ so that we have $b>-1/2$.
The second one requires
$c-b=(c-2b)(b+1)$ which means $c=2b+1$.
It can be verified that the other conditions
\eqref{eqn: condition1 on alpha-n for POP} and
\eqref{eqn: condition2 on alpha-n for POP}
for the choice of $\alpha_n$ are also satisfied for $c=2b+1$.
Hence $\mathcal{R}_n(\lambda)$, $n\geq1$,
satisfies the recurrence relation
\eqref{eqn: ttrr for Rn as POP with chain sequence} where, using
\eqref{eqn: ttrr for Rn from ttrr for Qn after dividing by lambda-1},
we have
\begin{align*}
d_{n+1}=
-\frac{1}{4}\frac{\tau_n}{\rho_n}\frac{p_{n+2}}{r_{n+1}}
=
\frac{1}{4}\frac{n(2b+n+1)}{(b+n)(b+n+1)},
\quad
n\geq1.
\end{align*}
The parameter sequence is given by $m_n=\frac{p_{n+1}}{2r_n}=
\frac{2b+n}{2(b+n)}$, $n\geq2$, with $m_1=1+\frac{\tau_1}{2\rho_1}=
\frac{2b+1}{2(b+1)}$.
Further, the Szeg\H{o} polynomials are given by
\begin{align*}
\phi_n(\lambda)=\frac{(2b+2)_{n}}{(b+1)_n}F(-n,b;&2b+1;1-\lambda)\\
&-\frac{n}{(b+n)}
\frac{(2b+2)_{n-1}}{(b+1)_{n-1}}F(-n+1,b;2b+1;1-\lambda),
\end{align*}
with the Verblunsky coefficients
\begin{align*}
\alpha_{n-1}=-(1+\tau_n\rho_n^{-1})=-\frac{b}{b+n},
\quad
n\geq1.
\end{align*}

We conclude with an observation.
From results illustrated above,
for $c=2b+1$ we have,
\begin{align}
\label{eqn: linear combination for Qn obtained in illustration}
\lefteqn{\frac{b}{2b+1}(\lambda-1)F(-n+1,b+1;2b+2;1-\lambda)}\nonumber\\
&&=
F(-n,b;2b+1;1-\lambda)-
F(-n+1,b;2b+1;1-\lambda).
\end{align}
Denoting the monic polynomials
\begin{align*}
\mathcal{R}_n(b;\lambda)
&=
\frac{(2b)_{n}}{(b)_n}F(-n,b,2b,1-\lambda),
\qquad n\geq1 \qquad\mbox{and}\\
\mathcal{\phi}_n(b;\lambda)
&=
\frac{(2b+1)_{n}}{(b+1)_n}F(-n,b+1,2b+1,1-\lambda),
\qquad n\geq1,
\end{align*}
it has been proved that (see the concluding remarks at the end of
\cite{Ranga-szego-polynomials-2010-AMS})
if $b>-1/2$, $\phi_n(b;\lambda)$
is a Szeg\H{o} polynomial with respect to the weight
function $[\sin{\theta/2}]^{2b}$ and in fact, is obtained from
Gegengauer polynomials using the Szeg\H{o} transformation.
Further, the reversed Szeg\H{o} polynomials
$\phi_n^{\ast}(b;\lambda)=
\lambda^n\overline{\phi_n(b;1/\bar{\lambda})}$
are given by
\begin{align*}
\phi_n^{\ast}(b;\lambda)=
\frac{(2b+1)_n}{(b+1)_n}F(-n,b;2b+1;1-\lambda).
\end{align*}
Thus, multiplying both sides of
\eqref{eqn: linear combination for Qn obtained in illustration}
by $\frac{(2b+1)_{n+1}}{(b)_{n+1}}$,
we observe that
\begin{align}
\label{eqn: definition obtain in illustration of POP}
\mathcal{R}_{n}(b+1;\lambda)=\frac{2b+n+1}{b}
\left[\frac{\phi_n^{\ast}(b;\lambda)-
\frac{2b+n+2}{b+n}\phi_{n-1}^{\ast}(b;\lambda)}
{\lambda-1}\right],
\quad n\geq1.
\end{align}
On the other hand, we obtain from
\cite[Theorem 5.1]{Ranga-szego-polynomials-2010-AMS}
\begin{align}
\label{eqn: definition usual of POP}
\mathcal{R}_{n}(b+1;\lambda)=
\left(\frac{2b+n}{b+n}\right)^{-1}[\phi_n(b+1;\lambda)+
\phi_n^{\ast}(b+1;\lambda)],
\quad n\geq1,
\end{align}
which follows the usual definition of a para-orthogonal polynomial
\cite{Jones-Njasad-Thron-Moment-OP-CF-1989-BLMS}.

For the purpose of illustration, we plot the distribution of the zeros of
$\mathcal{R}_{n}(b+1;\lambda)$,
$\mathcal{\phi}_{n}^{\ast}(b;\lambda)$,
$\mathcal{\phi}_{n-1}^{\ast}(b;\lambda)$,
$\mathcal{\phi}_{n}(b+1;\lambda)$
and
 $\mathcal{\phi}_{n}^{\ast}(b+1;\lambda)$
  for $n=12$ and $b=0.5$.
\begin{center}
\begin{figure}[ht!]
  \begin{subfigure}[b]{0.4\textwidth}
    \includegraphics[width=\textwidth]{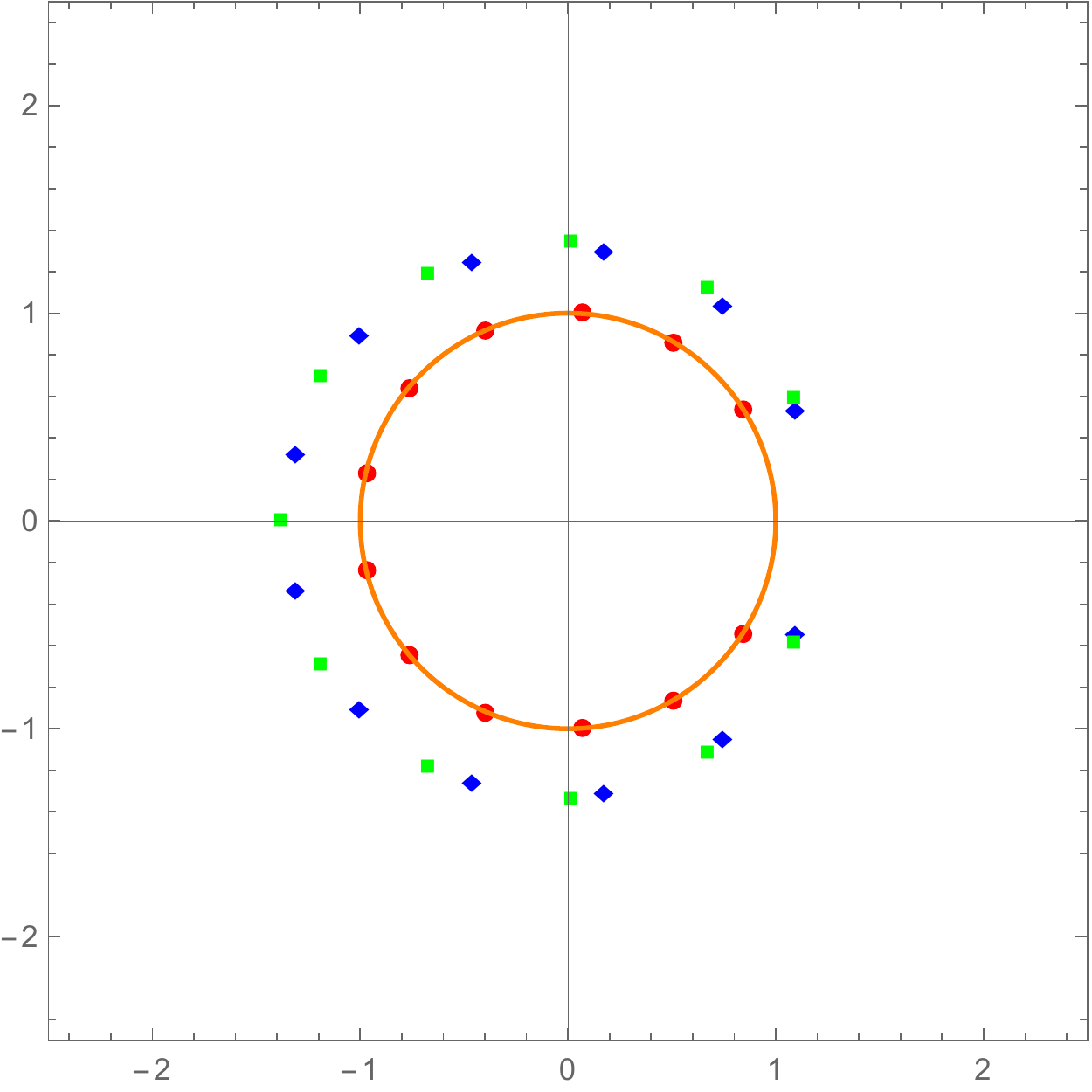}
    \caption{The case \eqref{eqn: definition obtain in illustration of POP}}
    \label{fig:different definition POP}
  \end{subfigure}
  \hfill
  \begin{subfigure}[b]{0.4\textwidth}
    \includegraphics[width=\textwidth]{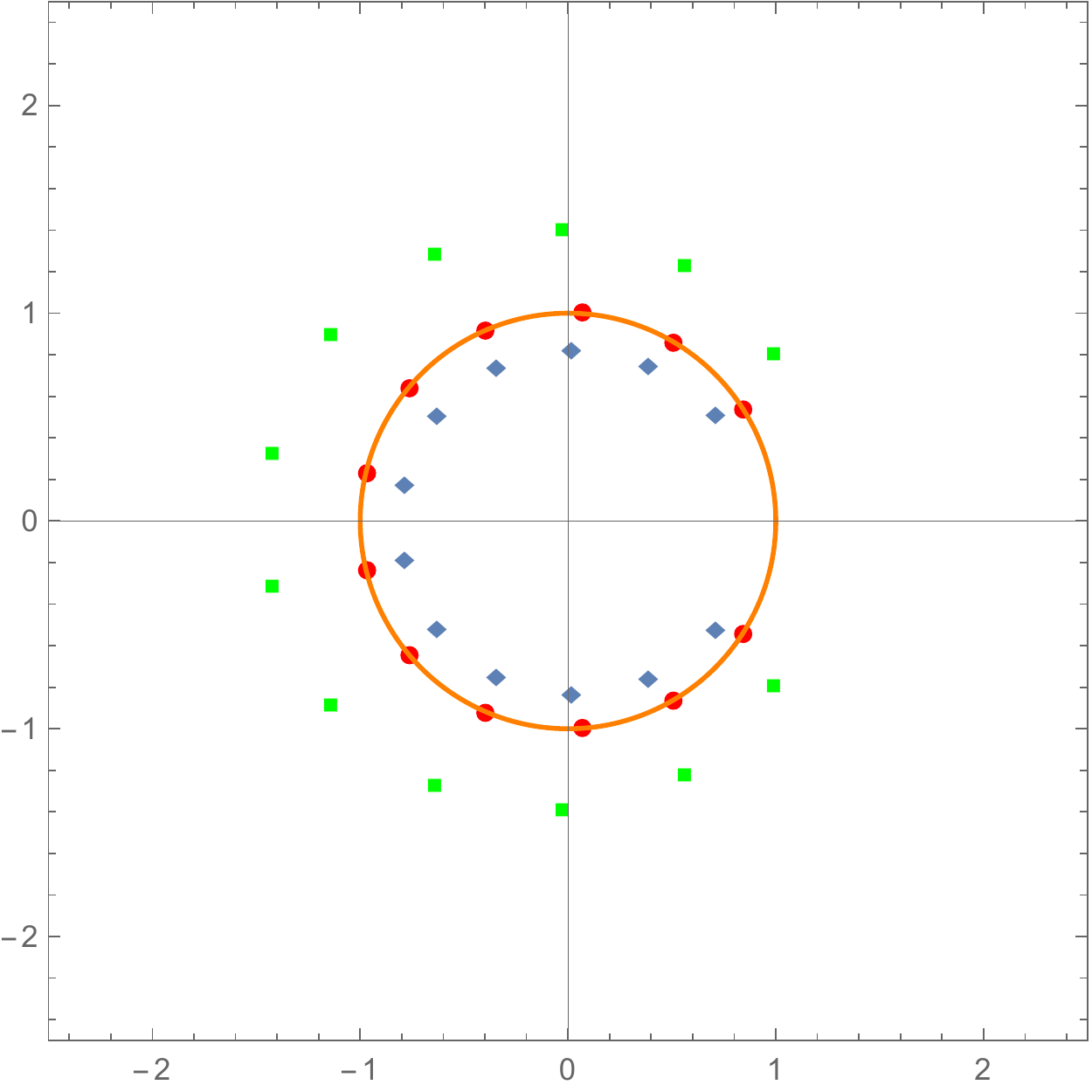}
    \caption{The case \eqref{eqn: definition usual of POP}}
    \label{fig:usual definition POP}
  \end{subfigure}
  \caption{The zeros of $\mathcal{R}_{n}(b+1;\lambda)$,
   $\mathcal{\phi}_{n}(b+1;\lambda)$,
   $\mathcal{\phi}_{n}^{\ast}(b+1;\lambda)$,
   $\mathcal{\phi}_{n}^{\ast}(b;\lambda)$,
   $\mathcal{\phi}_{n-1}^{\ast}(b;\lambda)$
  for $n=12$ and $b=0.5$.}
\end{figure}
\end{center}
In figures \eqref{fig:different definition POP} and \eqref{fig:usual definition POP},
the circular dots ($\bullet$ in red) are the zeros
of $\mathcal{R}_{12}(1.5,\lambda)$ lying on the unit circle,
the squares ({\tiny{$\blacksquare$}} in green) are of $\mathcal{\phi}_{11}^{\ast}(0.5;\lambda)$ and
$\mathcal{\phi}_{12}^{\ast}(1.5;\lambda)$ respectively
while the diamonds ({\tiny{$\blacklozenge$}} in blue) are of
 $\mathcal{\phi}_{12}^{\ast}(0.5;\lambda)$ and  $\mathcal{\phi}_{12}(1.5;\lambda)$ respectively.

\bibliographystyle{amsplain}

\begin{thebibliography}{99}

\bibitem{Alfaro-linear-combinations-JCAM-2010}
M. Alfaro, Francisco Marcell\'{a}n, Ana Pe\~{n}a\ and\ M. Luisa Rezola,
When do linear combinations of orthogonal polynomials yield
new sequences of orthogonal polynomials?,
J. Comput. Appl. Math.
{\bf 233} (2010), no.~6, 1446--1452.

\bibitem{Askey-discussion-Szego-paper-1982}
R. Askey,
Discussion of Szeg\"{o}'s paper ``Beitr\"{a}ge zur Theorie der Toeplitzschen Formen''.
In:R. Askey, editor. Gabor Szeg\"{o}.
Collected works. Vol. I. Boston, MA: Birkh\"{a}user; 1982;
p. 303--305

\bibitem{Askey-problems-on-SF-computations-1985}
R. Askey,
Some problems about special functions and computations,
Rend. Sem. Mat. Univ. Politec. Torino
{\bf 1985}, Special Issue, 1--22.

\bibitem{Ranga-Swami-POP-OPUC-ANM-2016}
C. F. Bracciali, A. Sri Ranga\ and\ A. Swaminathan,
Para-orthogonal polynomials on the unit circle
satisfying three term recurrence formulas,
Appl. Numer. Math. {\bf 109} (2016), 19--40.

\bibitem{Brezinski-Driver-ANM-2004-quasi-orthogonality}
C. Brezinski, K. A. Driver\ and\ M. Redivo-Zaglia,
Quasi-orthogonality with applications to some families
of classical orthogonal polynomials,
Appl. Numer. Math.
{\bf 48} (2004), no.~2, 157--168.


\bibitem{Ranga-Favard-type-OPUC-JAT-2014}
K. Castillo, M. S. Costa, A. Sri Ranga\ and\ D. O. Veronese,
A Favard type theorem for orthogonal polynomials on the
unit circle from a three term recurrence formula,
J. Approx. Theory {\bf 184} (2014), 146--162.

\bibitem{Ranga-OPUC-chain-sequence-JAT-2013}
M. S. Costa, H. M. Felix\ and\ A. Sri Ranga,
Orthogonal polynomials on the unit circle and chain sequences,
J. Approx. Theory {\bf 173} (2013), 14--32.

\bibitem{Chihara-quasi-orthogonalit-AMS-1957}
T. S. Chihara,
On quasi-orthogonal polynomials,
Proc. Amer. Math. Soc. {\bf 8} (1957), 765--767.

\bibitem{Delsarte-Genin-split-Levinson-IEEE-1986}
P. Delsarte\ and\ Y. V. Genin,
The split Levinson algorithm,
IEEE Trans. Acoust. Speech Signal Process.
{\bf 34} (1986), no.~3, 470--478.

\bibitem{Dickinson-quasi-orthogonality-AMS-1961}
D. Dickinson,
On quasi-orthogonal polynomials,
Proc. Amer. Math. Soc. {\bf 12} (1961), 185--194.

\bibitem{Draux-quasi-order-r-ITSF-2016}
A. Draux,
On quasi-orthogonal polynomials of order $r$,
Integral Transforms Spec. Funct.
{\bf 27} (2016), no.~9, 747--765.

\bibitem{Driver-Muldoon-common-zeros-Laguerre-JAT-2015}
K. Driver\ and\ M. E. Muldoon,
Common and interlacing zeros of families of Laguerre polynomials,
J. Approx. Theory {\bf 193} (2015), 89--98.

\bibitem{Fejer-quasi-orthogonality-1933}
L. Fej\'{e}r,
Mechanische Quadraturen mit positiven
Cotesschen Zahlen, Math. Z.
{\bf 37} (1933), 287--309.

\bibitem{Gibson-common-zeros-JAT-2000}
P. C. Gibson,
Common zeros of two polynomials in an orthogonal sequence,
J. Approx. Theory
{\bf 105} (2000), no.~1, 129--132.

\bibitem{Hendriksen-Njastad-biorthogonal-derivative-Rocky-1991}
E. Hendriksen\ and\ O. Nj\aa stad,
Biorthogonal Laurent polynomials with biorthogonal derivatives,
Rocky Mountain J. Math. {\bf 21} (1991), no.~1, 301--317.

\bibitem{Ismail-book}
M. E. H. Ismail,
{\it Classical and quantum orthogonal polynomials in one variable},
        reprint of the 2005 original,
        Encyclopedia of Mathematics and its Applications, 98,
        Cambridge Univ. Press, Cambridge, 2009.

\bibitem{Ismail-Masson-generalized-orthogonality-JAT-1995}
M. E. H. Ismail\ and\ D. R. Masson,
Generalized orthogonality and continued fractions,
J. Approx. Theory
{\bf 83} (1995), no.~1, 1--40.

\bibitem{Jones-Njasad-Thron-Moment-OP-CF-1989-BLMS}
        W. B. Jones, O. Nj\aa stad\ and\ W. J. Thron,
        Moment theory, orthogonal polynomials, quadrature, and continued fractions
        associated with the unit circle,
        Bull. London Math. Soc.
        {\bf 21} (1989), no.~2, 113--152.

\bibitem{Jones-Thron-strong-stieltjes-moment-AMS-1980}
W. B. Jones, W. J. Thron\ and\ H. Waadeland,
A strong Stieltjes moment problem,
Trans. Amer. Math. Soc.
{\bf 261} (1980), no.~2, 503--528.

\bibitem{Jordaan-mixed-recurrence-Acta-Hungarica-2010}
K. Jordaan\ and\ F. To\'{o}kos,
Mixed recurrence relations and interlacing of the zeros of some $q$-orthogonal
polynomials from different sequences,
Acta Math. Hungar. {\bf 128} (2010), no.~1-2, 150--164.

\bibitem{Marcellan-linear-combinations-1996-advances}
F. Marcell\'{a}n, F. Peherstorfer\ and\ R. Steinbauer,
Orthogonality properties of linear combinations
of orthogonal polynomials,
Adv. Comput. Math. {\bf 5} (1996), no.~4, 281--295.

\bibitem{Riesz-quasi-orthogonality}
M. Riesz,
Sur le probl\`{e}me des moments,
Troisi\`{e}me Note, Ark. Mat. Fys.
{\bf 17} (1923), 1--52.

\bibitem{Shohat-quasi-orthogonality-AMS-1937}
J. Shohat,
On mechanical quadratures, in particular, with positive coefficients,
Trans. Amer. Math. Soc.
{\bf 42} (1937), no.~3, 461--496.

\bibitem{Silva-Ranga-bounds-complex-zeros-JAT-2005}
A. P. da Silva\ and\ A. Sri Ranga,
Polynomials generated by a three term recurrence relation: bounds for complex zeros,
 Linear Algebra Appl.
 {\bf 397} (2005), 299--324.

\bibitem{Simon-book-Part-1}
B. Simon,
{\it Orthogonal polynomials on the unit circle. Part 1},
American Mathematical Society Colloquium Publications, 54, Part 1, American Mathematical Society, Providence, RI, 2005.

\bibitem{Ranga-szego-polynomials-2010-AMS}
A. Sri Ranga,
Szeg\H o polynomials from hypergeometric functions,
Proc. Amer. Math. Soc.
{\bf 138} (2010), no.~12, 4259--4270.

\bibitem{Szego-book}
G. Szeg\"{o},
{\it Orthogonal polynomials}, fourth edition,
American Mathematical Society, Providence, RI, 1975.

\bibitem{Koepf-mixed-recurrence-ANM-2018}
D. D. Tcheutia, A. S. Jooste\ and\ W. Koepf, Mixed recurrence equations and interlacing properties for zeros of sequences of classical $q$-orthogonal polynomials, Appl. Numer. Math. {\bf 125} (2018), 86--102.

\bibitem{Temme-biorthogonal-Constapprx-1986}
N. M. Temme,
Uniform asymptotic expansion for a class of polynomials biorthogonal on the unit circle,
Constr. Approx. {\bf 2} (1986), no.~4, 369--376.

\bibitem{Wong-1st-2nd-kind-POP-JAT-2007}
M. L. Wong, First and second kind paraorthogonal polynomials and their zeros, J. Approx. Theory {\bf 146} (2007), no.~2, 282--293.

\bibitem{Zhedanov-ortho-polygons-JAT-1999}
A. Zhedanov, On the polynomials orthogonal on regular polygons, J. Approx. Theory {\bf 97} (1999), no.~1, 1--14. MR1676314

\end{thebibliography}

\end{document}